\documentclass[a4paper,11pt,twoside]{amsart}

\usepackage{amsmath}
\usepackage{amsthm}
\usepackage{amsfonts}
\usepackage{amssymb}
\usepackage{mathrsfs}
\usepackage{graphicx}
\usepackage{enumitem}
\usepackage{url}

\usepackage{tikz}

\usepackage[labelformat=simple]{subcaption}

\usetikzlibrary{arrows}
\usetikzlibrary{decorations.markings}

\theoremstyle{plain}
\newtheorem{lemma}{Lemma}[section]
\newtheorem{prop}[lemma]{Proposition}
\newtheorem{theo}[lemma]{Theorem}
\newtheorem{coro}[lemma]{Corollary}
\theoremstyle{remark}
\newtheorem{rem}[lemma]{Remark}
\newtheorem*{notat}{Notation}
\theoremstyle{definition}
\newtheorem{definition}[lemma]{Definition}
\newtheorem{ex}[lemma]{Example}

\newcommand{\B}{\mathcal{B}}
\newcommand{\op}{\textnormal{op}}
\newcommand{\id}{\mathrm{Id}}

\newcommand{\N}{\mathbb{N}}
\newcommand{\tq}{\;/\;}
\renewcommand{\S}{\mathcal{S}}
\newcommand{\arr}{-triangle 45}
\newcommand{\fintop}{\text{\bf FinTop}}
\newcommand{\fintopzero}{\text{\bf FinTop}_{\bf 0}}
\DeclareMathOperator{\im}{Im}
\DeclareMathOperator{\aut}{Aut}
\DeclareMathOperator{\mnl}{mnl}
\DeclareMathOperator{\rel}{\textnormal{rel}}
\newcommand{\Top}{\textnormal{\bf Top}}
\newcommand{\ev}{\textnormal{ev}}

\newcommand{\pos}{\textnormal{\bf Pos}}
\newcommand{\cat}{\textnormal{\bf Cat}}

\begin{document}
\title[Fibrations between finite topological spaces]{Fibrations between finite topological spaces}

\author{Nicolas Cianci}
\address{Facultad de Ciencias Exactas and Naturales \\ Universidad Nacional de Cuyo \\ Mendoza, Argentina.}
\email{nicocian@gmail.com}

\author{Miguel Ottina}
\address{Facultad de Ciencias Exactas and Naturales \\ Universidad Nacional de Cuyo \\ Mendoza, Argentina.}
\email{miguelottina@gmail.com}




\thanks{Research partially supported by grant M044 (2016--2018) of SeCTyP, UNCuyo. The first author was also partially supported by a CONICET doctoral fellowship.}

\begin{abstract}
We study Hurewicz fibrations between finite T$_0$--spaces from a combinatorial viewpoint and give strong conditions that a continuous map between finite T$_0$--spaces must satisfy in order to be a Hurewicz fibration. We also show that there exists a strong relationship between Hurewicz fibrations between finite T$_0$--spaces and Grothendieck bifibrations. Finally we give several interesting examples that illustrate this theory and show that many of the assumptions of our results are necessary.
\end{abstract}

\maketitle

\section{Introduction}

In this article we will study the combinatorial aspects of the Hurewicz fibrations between finite T$_0$--spaces. Some of the results we obtain can be easily extended to Hurewicz fibrations between Alexandroff T$_0$--spaces. However, we will not do explicit mention of these facts in this work since we are interested in studying this problem in the context of finite topological spaces, where more things can be said.

We mention that in \cite{cianci2019combinatorial} we gave a complete combinatorial characterization of the cofibrations between finite topological spaces, while in \cite{cianci2019classification} we obtained a classification theorem for fiber bundles over Alexandroff spaces with T$_0$ fiber. However, the same problem for fibrations seems a much harder task.

\medskip

This article is organized as follows. In section 3 we give two preliminary results about continuity of maps whose codomain is an Alexandroff space. In section 4 we extend some of Stong's definitions to the category $\fintopzero/B$ for a given finite T$_0$--space $B$, and we prove that many of the results of the theory of Stong, as well as the results of \cite{cianci2019combinatorial} regarding bp--retracts, can be generalized to this category. Then, in section 5, we develop results which relate the beat points of fibrations between finite T$_0$--spaces with the beat points of its base space, total space and fibers, as well as the behaviour of the Hurewicz fibrations in the presence or absense of such beat points.

In section 6, we study the regularity of Hurewicz fibrations between finite T$_0$--spaces and we apply the results obtained to give a strong relationship between Hurewicz fibrations and Grothendieck bifibrations which, under specific and strong conditions over the base space, suffices to completely characterize the Hurewicz fibrations between finite T$_0$--spaces over that base space.

Finally, in the last section of the article we give examples that show that many of the assumptions of our results are indeed necessary.

\section{Preliminaries}

In this section we will recall several definitions and results that will be needed to develop the results of this article.

First of all we will introduce some notation that will be used throughout the article.

\begin{notat}\ 
\begin{itemize}
\item The topological space $[0,1]$ with the usual topology will be denoted by $I$.
\item Let $X$ be a topological space. For each $t\in I$ we define the map $i_t\colon X\to X\times I$ by $i_t(x)=(x,t)$.
\item Let $X$ and $Y$ be topological spaces. The space of continuous maps from $X$ to $Y$ with the compact-open topology will be denoted by $Y^X$.
\item Let $X$ and $Y$ be topological spaces. We define the \emph{evaluation map} as the map $\ev\colon Y^X \times X \to Y$ given by $\ev(f,x)=f(x)$. We also define, for each $a\in X$, the map $\ev_a\colon Y^X \to Y$ given by $\ev_a(f)=f(a)$.
\item Let $X$, $Y$ and $Z$ be topological spaces. Let $g\colon Y\to Z$ be a continuous map. We define the map $g^X\colon Y^X \to Z^X$ by $g^X(\alpha)=g\circ\alpha$.
\item The category of topological spaces and continuous maps will be denoted by $\Top$. 
In addition, if $B$ is a topological space, $\Top/B$ will denote the slice category of $\Top$ over $B$, that is the category whose objects are the continuous maps with codomain $B$ and whose morphisms are given by commutative triangles.
\item Similarly, the category of finite topological spaces and continuous maps will be denoted by $\fintop$. The full subcategory of $\fintop$ whose objects are the finite T$_0$--spaces will be denoted by $\fintopzero$. In addition, if $B$ is a finite T$_0$--space, $\fintopzero/B$ will denote the slice category of $\fintopzero$ over $B$.
\item The category of small categories will be denoted by $\cat$. The category of posets and order-preserving morphisms will be denoted by $\pos$. We will frequently consider a poset as a small category with arrows given by the order relation and therefore $\pos$ will be sometimes regarded as a subcategory of $\cat$.
\end{itemize}
\end{notat}

We also recall some definitions regarding homotopies in a slice category of $\Top$.

\begin{definition}\label{defi:rdf_sobre_B}
Let $B$ be a topological space and let $p\colon E\to B$ and $p'\colon E'\to B$ be two continuous maps considered as objects of $\Top/B$. Let $f\colon p\to p'$ be an arrow over $B$. We say that $f$ is a \emph{homotopy equivalence (over $B$)} if there exists an arrow $g\colon p'\to p$ over $B$, which is called \emph{homotopy inverse of $f$ (over $B$)} such that $gf\simeq \id_p$ and $fg\simeq\id_{p'}$. In that case, we say that $p$ and $p'$ are \emph{homotopy equivalent over $B$} or \emph{fiber homotopy equivalent}. 

Let $A\subseteq E$ be a subspace and let $i\colon A\to E$ be the inclusion map. Consider the continuous map $p|=pi\colon A\to B$ as an object over $B$. 

We will say that $p|$ is a \emph{strong deformation retract} of $p$ if there exists a retraction $r\colon p\to p|$, such that $\id_p\simeq ir\ (\rel A)$.

Note that if $p|$ is a strong deformation retract of $p$ in $\Top/B$, then $A$ is a strong deformation retract of $E$ in $\Top$. 
\end{definition}

\subsection{Finite topological spaces}

Let $X$ be an Alexandroff topological space. For each $a\in X$ we define $U^X_a$ as the intersection of all the open subsets of $X$ that contain $a$, and it will be also denoted by $U_a$ when the space in which it is considered is clear from the context. Clearly, $U_a$ is the smallest open subset of $X$ that contains $a$ and $\{U_x\tq x\in X\}$ is a basis for the topology of $X$. We define a preorder relation in $X$ by $x_1\leq x_2$ if and only if $U_{x_1}\subseteq U_{x_2}$. It follows that, for each $a\in X$, $U_a=\{x\in X \tq x\leq a\}$.

The preorder given above is a partial order if and only if $X$ is a T$_0$--space. Moreover, this defines a one-to-one correspondence between Alexandroff topologies in a set $X$ and preorder relations in $X$, which was first stated by Alexandroff \cite{alexandroff1937diskrete}. Moreover, a map between Alexandroff spaces is continuous if and only if it is an order-preserving morphism between the associated preordered sets.

If $X$ is an Alexandroff space and $a\in X$ it is standard to define 
\begin{itemize}
\item $F^X_a=\overline{\{a\}}=\{x\in X \tq x\geq a\}$,
\item $\widehat{U}^X_a=U^X_a-\{a\}=\{x\in X \tq x< a\}$, and
\item $\widehat{F}^X_a=F^X_a-\{a\}=\{x\in X \tq x>a\}$.
\end{itemize}
where the superindex $X$ is usually omitted from the notation when the Alexandroff space in which these sets are considered is clear from the context.

\smallskip

Now let $X$ be a finite T$_0$--space and let $a\in X$. We say that $a$ is a \emph{down beat point} of $X$ if the set $\widehat{U}_a$ has a maximum element and we say that $a$ is an \emph{up beat point} of $X$ if the set $\widehat{F}_a$ has a minimum element. We say that $a$ is a \emph{beat point} of $X$ if it is either a down beat point or an up beat point \cite{barmak2011algebraic, may2016finite, stong1966finite}. In addition, we say that $X$ is a \emph{minimal space} if $X$ does not have beat points. A subspace $A\subseteq X$ is a \emph{core} of $X$ if $A$ is a minimal space and a strong deformation retract of $X$. Observe that one can obtain a core of $X$ by successively removing its beat points.

Stong proved in \cite{stong1966finite} that if $X$ is a finite T$_0$--space and $a$ is a beat point of $X$, then $X-\{a\}$ is a strong deformation retract of $X$. In addition, he proved that two cores of a finite T$_0$--space $X$ are homeomorphic and that two finite T$_0$--spaces are homotopy equivalent if and only if they have homeomorphic cores. He also obtained the following result in \cite[Theorem 3]{stong1966finite} and its proof.
  
\begin{prop}[{\cite[p.330]{stong1966finite}}]\label{prop_stong}
Let $X$ be a finite T$_0$--space and let $f\colon X\to X$ be a continuous map.  
\begin{enumerate}
\item If $X$ does not have down beat points and $f\leq\id_X$, then $f=\id_X$.
\item If $X$ does not have up beat points and $f\geq\id_X$, then $f=\id_X$.
\item $X$ does not have beat points and $f\simeq \id_X$, then $f=\id_X$.
\end{enumerate}
\end{prop}

Stong also proves that finite spaces are exponentiable \cite[Lemma 1]{stong1966finite} and obtains as a corollary that if $f,g\colon X\to Y$ are continuous maps between finite spaces such that $f(x)\leq g(x)$ for all $x\in X$ then $f$ and $g$ are homotopic relative to the subset $\{x\in X:f(x)=g(x)\}$. We will give a generalization of this result in \ref{prop:f_leq_g_son_hom}.

The following is a generalization of \cite[Theorem 7(b)]{stong1966finite} for maps over a finite space $B$. We will omit its proof since it is analogous to that of the original result of Stong.

\begin{prop}\label{prop:fence}
Let $X$, $Y$ and $B$ be finite topological spaces, let $A\subseteq X$ and let $p\colon X\to B$ and $q\colon Y\to B$ be continuous maps, considered as objects over $B$. Let $f,g\colon X\to Y$ be continuous maps such that $f|_A=g|_A$ and $qf=qg=p$, considered as arrows over $B$ from $p$ to $q$. 
\begin{enumerate}
\item If $f\leq g$, then $f$ is fiber homotopic to $g$ relative to $A$.
\item If $f$ and $g$ are fiber homotopic relative to $A$, then there exist $n\in\N$ and continuous maps $f_0,f_1,\ldots,f_n$ over $B$ from $p$ to $q$ whose restrictions to $A$ coincide with those of $f$, such that $f=f_0\leq f_1\geq \cdots\leq f_n=g$.
\end{enumerate}
\end{prop}

\subsection*{bp-retracts} In this subsection we will recall some definitions and results from \cite{cianci2019combinatorial}.

\begin{definition}[{\cite[Definition 3.1]{cianci2019combinatorial}}]
Let $X$ be a finite T$_0$--space and let $A\subseteq X$. We will say that $A$ is a \emph{dbp--retract} (resp.\ \emph{ubp--retract}) of $X$ if $A$ can be obtained from $X$ by successively removing down beat points (resp.\ up beat points), that is, if there exist $n\in\N_0$ and a sequence $X=X_0\supseteq X_1\supseteq\cdots\supseteq X_n=A$ of subspaces of $X$ such that, for all $i\in\{1,\ldots,n\}$, the space $X_i$ is obtained from $X_{i-1}$ by removing a single down beat point (resp.\ up beat point) of $X_{i-1}$.

We will say that $A$ is a \emph{bp--retract} of $X$ if $A$ is either a dbp--retract or a ubp--retract of $X$.
\end{definition}

\begin{theo}[{\cite[Theorem 3.5]{cianci2019combinatorial}}] \label{theo:dbpr_equivalences}
Let $X$ be a finite T$_0$--space, let $A$ be a subspace of $X$ and let $i\colon A\to X$ be the inclusion map. Then, the following propositions are equivalent:
\begin{enumerate}
\item $A$ is a dbp--retract of $X$.
\item There exists a continuous map $f\colon X\to X$ such that $f\leq\id_X$, $f^{2}=f$ and $f(X)=A$.
\item There exists a unique continuous map $f\colon X\to X$ such that $f\leq\id_X$, $f^{2}=f$ and $f(X)=A$.
\item There exists a retraction $r\colon X\to A$ of $i$ such that $ir\leq\id_X$.
\item There exists a unique retraction $r\colon X\to A$ of $i$ such that $ir\leq\id_X$.
\end{enumerate}
\end{theo}

\begin{coro}[{\cite[Corollary 3.11]{cianci2019combinatorial}}]
\label{coro:subseteq_iff_f1_leq_f2}
Let $X$ be a finite T$_0$--space and, for $k=1,2$, let $f_k\colon X\to X$ be a continuous map such that $f_k\circ f_k=f_k$ and $f_k\leq \id_X$. Then $f_1\leq f_2$ if and only if $f_1(X)\subseteq f_2(X)$.
\end{coro}

\subsection*{Topological Grothendieck construction and fiber bundles}

In \cite{cianci2019classification}, we introduced the notion of topological Grothendieck construction of a (covariant) functor from a preordered set to the category of topological spaces, which coincides with the Grothendieck construction if the functor takes values in the subcategory of Alexandroff spaces. In this subsection we will recall the definition of topological Grothendieck construction and some results of \cite{cianci2019classification} that will be needed later.

\begin{definition}[{\cite[Definition 3.1]{cianci2019classification}}]
Let $B$ be an Alexandroff space considered as a preordered set and let $D\colon B\to\Top$ be a functor. We define 
\begin{displaymath}
\int D = \bigcup\limits_{b\in B}\{b\}\times D(b).
\end{displaymath}
For each $b\in B$ and for each $V$ open subset of $D(b)$ we define 
\begin{displaymath}
J(b,V)=\bigcup\limits_{v\in U_b}\{v\}\times D(v\leq b)^{-1}(V).
\end{displaymath}
Let $\B=\{J(b,V):b\in B\text{ and $V$ is an open subset of $D(b)$}\}$. The set $\B$ is a basis for a topology on $\int D$. We consider $\int D$ as a topological space with the topology generated by $\B$. The topological space $\int D$ will be called  the \emph{topological Grothendieck construction} of $D$.
\end{definition}

\begin{prop}[{\cite[Remark 3.8]{cianci2019classification}}]
\label{rema:int_is_functor}
Let $B$ be an Alexandroff space, let $D,E\colon B\to\Top$ be functors and let $\eta\colon D\Rightarrow E$ be a natural transformation. 
Let $\eta_*\colon\int D\to\int E$ be defined by $\eta_*(b,x)=(b,\eta_b(x))$. Then $\eta_*$ is continuous and hence a map over $B$.
\end{prop}

\begin{prop}[{\cite[Proposition 3.9]{cianci2019classification}}]
\label{prop:int_D_is_fiber_bundle}
Let $B$ be a connected Alexandroff space and let $D\colon B\to \Top$ be a morphism-inverting functor. Then $\pi_B\colon \int D\to B$ is a fiber bundle over $B$ with fiber $D(b_0)$ for any $b_0\in B$.
\end{prop}

\begin{theo}[{\cite[Theorem 3.20]{cianci2019classification}}]
\label{theo:classification_fiber_bundles_Grothendieck_construction}
Let $B$ be a connected Alexandroff space and let $F$ be any T$_0$--space. Then there exists a canonical bijection between isomorphism classes of fiber bundles over $B$ with fiber $F$ and isomorphism classes of functors from $B$ to $\aut(F)$. This bijection is induced by the canonical representation and its inverse is induced by the topological Grothendieck construction.
\end{theo}

\begin{coro}[{\cite[Corollary 3.21]{cianci2019classification}}] \label{coro_fiber_bundle_with_simply_connected_base} 
Let $B$ be a simply connected Alexandroff space, let $F$ be any T$_0$--space and let $p$ be a fiber bundle over $B$ with fiber $F$. Then $p$ is a trivial fiber bundle.
\end{coro}

\subsection*{Fibrations}

Recall that a continuous map $p\colon E\to B$ between topological spaces is called a \emph{(Hurewicz) fibration} if it has the homotopy lifting property with respect to all topological spaces, that is, if for all topological spaces $X$ and continuous maps $f\colon X\to E$ and $H\colon X\times I\to B$ such that $Hi_0=pf$ there exists a continuous map $\widetilde H\colon X\times I \to E$ such that $p\widetilde H = H$ and $\widetilde Hi_0=f$.

Similarly, a continuous map $p\colon E\to B$ between topological spaces is called a \emph{Serre fibration} if it has the homotopy lifting property with respect to the $n$--dimensional disks $D^n$ for all $n\in\N_0$.

It is well known that compositions, products, pullbacks and retracts of fibrations are fibrations. Moreover, if $p\colon E\to B$ is a fibration and $X$ is an exponentiable space, then $p^{X}\colon E^{X}\to B^{X}$ is also a fibration. In particular, if $p\colon E\to B$ is a fibration and $X$ is a finite space, then $p^{X}\colon E^{X}\to B^{X}$ is also a fibration.

\begin{definition}
Let $E$ and $B$ be topological spaces and let $p\colon E\to B$ be a continuous map. Let $E\times_p B^{I}=\{(e,\gamma)\in E\times B^I \tq \gamma(0)=p(e) \}$ and let $\pi_E\colon E\times_p B^{I} \to E$ and $\pi_{B^I}\colon E\times_p B^{I} \to B^I$ denote the corresponding projection maps.
A \emph{path-lifting map for $p$} is a continuous map $\Lambda\colon E\times_p B^{I}\to E^{I}$ such that $\ev_0\Lambda=\pi_E$ and $p^{I}\Lambda=\pi_{B^{I}}$.
\end{definition}

From the exponential law it follows that a continuous map $p\colon E\to B$ is a (Hurewicz) fibration if and only if $p$ admits a path-lifting map.

\begin{definition}\label{defi:gamma_[0,t]}
Let $X$ be a topological space, let $\gamma\colon I \to X$ be a path in $X$ and let $t_0,t_1\in I$ be such that $0\leq t_0\leq t_1\leq 1$. We define the map $\gamma_{[t_0,t_1]}\colon I\to X$ by $\gamma_{[t_0,t_1]}(s)=\gamma(t_0+(t_1-t_0)s)$ for all $s\in I$. Note that $\gamma_{[t_0,t_1]}$ is a continuous map and hence, a path in $X$. 
\end{definition}

\begin{definition}\label{defi:eta}
Let $X$ be an Alexandroff space and let $x,y\in B$ be such that $x\leq y$. We define the path $\eta(x\leq y)\colon I\to B$ by
\begin{displaymath}
\eta(x\leq y)(t)=\begin{cases} x&\text{ if $t<1$,}\\ y&\text{ if $t=1$.}\end{cases}
\end{displaymath}
We define, on the other hand, the path $\eta(y\geq x)\colon I\to B$ by
\begin{displaymath}
\eta(y\geq x)(t)=\begin{cases} y&\text{ if $t=0$,}\\ x&\text{ if $t>0$.}\end{cases}
\end{displaymath}
Thus, the path $\eta(y\geq x)$ is the inverse path of $\eta(x\leq y)$.
\end{definition}

\subsection*{Grothendieck fibrations}

\begin{definition}
Let $E$ and $B$ be small categories and let $p\colon E\to B$ be a funtor.
\begin{itemize}
\item Let $f\colon e'\to e$ be an arrow of $E$. We say that $f$ is \emph{cartesian (with respect to $p$)} if for every arrow $g\colon e''\to e$ of $E$ and every arrow $h\colon p(e'')\to p(e')$ of $B$ such that $p(g)=p(f)h$ there exists a unique arrow $\tilde{h}\colon e''\to e'$ such that $f\tilde{h}=g$ and $p\big(\tilde{h}\big)=h$.
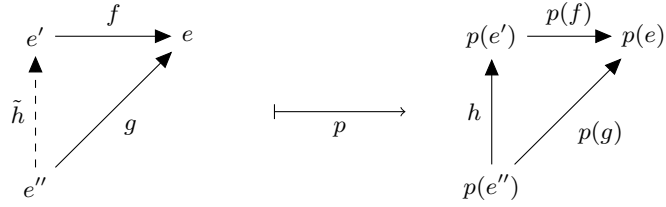
\begin{center}
\begin{tikzpicture}[x=2cm,y=2cm]
\tikzstyle{every node}=[font=\footnotesize]
\draw (0,0) node (e''){$e''$};
\draw (0,1) node (e'){$e'$};
\draw (1,1) node (e){$e$};
\draw[\arr] (e')->(e) node [midway,above] {$f$};
\draw[\arr] (e'')->(e) node [midway,below right] {$g$};
\draw[dashed, \arr] (e'')->(e') node [midway,left] {$\tilde{h}$};
\draw (1.5,0.5) node(s){};
\draw (2.5,0.5) node(t){};
\draw[|->] (s)--(t) node[midway, below]{$p$};
\draw (3,0) node (pe''){$p(e'')$};
\draw (3,1) node (pe'){$p(e')$};
\draw (4,1) node (pe){$p(e)$};
\draw[\arr] (pe')->(pe) node [midway,above] {$p(f)$};
\draw[\arr] (pe'')->(pe) node [midway,below right] {$p(g)$};
\draw[\arr] (pe'')->(pe') node [midway,left] {$h$};
\end{tikzpicture}
\end{center}
\item Let $f\colon e\to e'$ be an arrow of $E$. We say that $f$ is \emph{cocartesian (with respect to $p$)} if for every arrow $g\colon e\to e''$ of $E$ and every arrow $h\colon p(e')\to p(e'')$ of $B$ such that $p(g)=hp(f)$ there exists a unique arrow $\tilde{h}\colon e'\to e''$ such that $\tilde{h}f=g$ and $p\big(\tilde{h}\big)=h$.
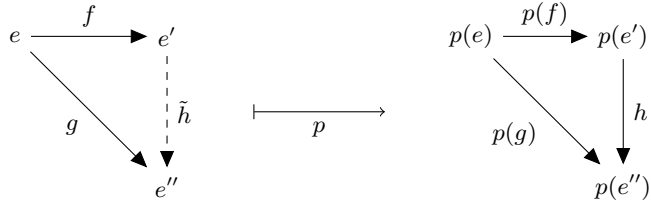
\begin{center}
\begin{tikzpicture}[x=2cm,y=2cm]
\tikzstyle{every node}=[font=\footnotesize]
\draw (1,0) node (e''){$e''$};
\draw (0,1) node (e){$e$};
\draw (1,1) node (e'){$e'$};
\draw[\arr] (e)->(e') node [midway,above] {$f$};
\draw[\arr] (e)->(e'') node [midway,below left] {$g$};
\draw[dashed, \arr] (e')->(e'') node [midway,right] {$\tilde{h}$};
\draw (1.5,0.5) node(s){};
\draw (2.5,0.5) node(t){};
\draw[|->] (s)--(t) node[midway, below]{$p$};
\draw (4,0) node (pe''){$p(e'')$};
\draw (3,1) node (pe){$p(e)$};
\draw (4,1) node (pe'){$p(e')$};
\draw[\arr] (pe)->(pe') node [midway,above] {$p(f)$};
\draw[\arr] (pe)->(pe'') node [midway,below left] {$p(g)$};
\draw[\arr] (pe')->(pe'') node [midway,right] {$h$};
\end{tikzpicture}
\end{center}
\end{itemize}
\end{definition}

\begin{definition}\label{defi:fib_grothendieck}
Let $E$ and $B$ be small categories and let $p\colon E\to B$ be a functor. We say that $p$ is a \emph{Grothendieck fibration} if for all $e\in E$, all $b\in B$ and all arrows $f\colon b\to p(e)$, there exists $e'\in p^{-1}(b)$ and a cartesian arrow $\tilde{f}\colon e'\to e$ such that $p\big(\tilde{f}\big)=f$. The arrow $\tilde{f}$ is called \emph{cartesian lift of $f$ to $e$}. 

Dually, we say that $p$ is a \emph{Grothendieck opfibration} if $p^{\op}$ is a Grothendieck fibration. In other words, $p$ is a Grothendieck opfibration if for all $e\in E$, all $b\in B$ and all arrows $f\colon p(e)\to b$, there exists $e'\in p^{-1}(b)$ and a cocartesian arrow $\tilde{f}\colon e\to e'$ such that $p\big(\tilde{f}\big)=f$. The arrow $\tilde{f}$ is called \emph{cocartesian lift of $f$ from $e$}. 

We say that $p$ is a \emph{Grothendieck bifibration} if $p$ is both a Grothendieck fibration and a Grothendieck opfibration.
\end{definition}

\begin{definition}
Let $E$ and $B$ be small categories and let $p\colon E\to B$ be a functor. A \emph{cleavage (for $p$)} is a map $\phi$ that assigns to each $e\in E_0$ and each arrow $f\colon b\to p(e)$ in $B$, a cartesian lift $\phi^{f}_e$ of $f$ to $e$. We say that a cleavage $\phi$ is \emph{closed} if it preserves identity maps and compositions, that is, $\phi_e^{\id_{p(e)}}=\id_e$ for all $e\in E$ and $\phi^{fg}_e=\phi_e^{f}\phi^{g}_{e'}$ for all $e\in E_0$, for all arrows $f\colon b'\to p(e)$ and $g\colon b''\to b'$ in $B$, where $e'$ is the domain of $\phi^{f}_e$.

Dually, an \emph{opcleavage (for $p$)} is a map $\phi$ that assigns to each $e\in E_0$ and each arrow $f\colon p(e)\to b$ in $B$, a cocartesian lift $\phi^{f}_e$ of $f$ from $e$. We say that the opcleavage $\phi$ is \emph{closed} if it preserves identity maps and compositions, that is, $\phi_e^{\id_{p(e)}}=\id_e$ for all $e\in E$ and $\phi^{gf}_e=\phi_{e'}^{g}\phi^{f}_{e}$ for all $e\in E_0$ and for all arrows $f\colon b\to b'$ and $g\colon b'\to b''$ in $B$, where $e'$ is the codomain of $\phi^{f}_e$. 
\end{definition}

Observe that assuming the axiom of choice, a functor between small categories is a Grothendieck fibration if and only if it admits a cleavage and it is a Grothendieck opfibration if and only if it admits an opcleavage.

\begin{definition}
A \emph{split (Grothendieck) fibration} is a pair $(p,\phi)$ where $p$ is a Grothendieck fibration and $\phi$ is a cleavage for $p$.

Dually, a \emph{split (Grothendieck) opfibration} is a pair $(p,\phi)$ where $p$ is a Grothendieck opfibration and $\phi$ is an opcleavage for $p$.
\end{definition}

\begin{theo}\label{theo:pi_B_iso_p_sobre_B}
Let $E$ and $B$ be small categories and let $p\colon E\to B$ be a split Grothendieck fibration. Then there exists a contravariant functor $F_p\colon B\to \cat$ such that the canonical projection $\pi^{F_p}_B\colon \int F_p\to B$ is an object over $B$ isomorphic to $p$. 
\end{theo}

\section{Preliminary results}

In this section we will prove two propositions that will allow us to deduce the continuity of a map whose codomain is an Alexandroff space from the continuity of other that is comparable with the first one. These propositions can be considered as `pasting lemmas' for maps of this type. They will be used in the following sections to prove that certain path lifting maps are continuous.

To this end, we will use the following lemmas.

\begin{lemma}\label{lemm:KVW_abto}
Let $X$ be a topological space, let $V\subseteq X$ be an open subspace, let $K\subseteq X$ be a closed subspace, and let $W\subseteq V$ be such that $K\cap W$ is open in $K$. Then
\[
(K^{c}\cap V)\cup (K\cap W)
\]
is open in $X$.
\end{lemma}

\begin{proof}
Since $K\cap W$ is open in $K$, there exists an open subset $U$ of $X$ such that $K\cap W=K\cap U$. In particular, as $W\subseteq V$, we have that $K\cap W=K\cap W\cap V=K\cap U\cap V$.
Hence,
\begin{align*}
(K^{c}\cap V)\cup(K\cap W)&=(K^{c}\cap V)\cup(K\cap U\cap V)=\\&=(K^{c}\cap V)\cup(K^{c}\cap U\cap V)\cup(K\cap U\cap V)=(K^{c}\cap V)\cup (U\cap V).
\end{align*}

Since $K^{c}$, $U$ and $V$ are open subsets of $X$, the result follows.
\end{proof}

\begin{lemma}\label{lemm:KVW_cerrado}
Let $X$ be a topological space, let $V\subseteq X$ be an open subspace, let $K\subseteq X$ be a closed subspace, and let $W\subseteq K$ be such that $V\cap W$ is closed in $V$. Then
\[
(V^{c}\cap K)\cup (V\cap W)
\]
is closed in $X$.
\end{lemma}

\begin{proof}
The proof of this lemma is analogous to the proof of \ref{lemm:KVW_abto}.
\end{proof}

\begin{prop}\label{prop:f_leq_g_g_cont}
Let $X$ be a topological space and let $Y$ be an Alexandroff space. Let $K\subseteq X$ be a closed subspace and let $f,g\colon X\to Y$ be two maps. If
\begin{enumerate}
\item $f$ is continuous,
\item $f\leq g$,
\item $f|_{K^{c}}=g|_{K^{c}}$, and
\item $g|_K$ is continuous
\end{enumerate}
then $g$ is continuous.
\end{prop}
\begin{proof}
Let $U\subseteq Y$ be an open subset.  
We have that
\[g^{-1}(U)=(K^{c}\cap g^{-1}(U))\cup(K\cap g^{-1}(U))=(K^{c}\cap f^{-1}(U))\cup(K\cap g^{-1}(U)).\]

It is clear that $f^{-1}(U)$ is open in $X$.
Now, as $f\leq g$ and $U$ is open, then $g^{-1}(U)\subseteq f^{-1}(U)$. On the other hand, $K\cap g^{-1}(U)=(g|_K)^{-1}(U)$ is open in $K$. Applying \ref{lemm:KVW_abto}, it follows that $g^{-1}(U)$ is open in $X$. Hence, $g$ is continuous. 
\end{proof}

\begin{prop}\label{prop:f_leq_g_f_cont}
Let $X$ be a topological space and let $Y$ be an Alexandroff space. Let $V\subseteq X$ be an open subspace and let $f,g\colon X\to Y$ be two maps. If
\begin{enumerate}
\item $g$ is continuous,
\item $f\leq g$,
\item $f|_{V^{c}}=g|_{V^{c}}$, and
\item $f|_V$ is continuous
\end{enumerate}
then $f$ is continuous.
\end{prop}
\begin{proof}
This proposition can be proved in a similar way as \ref{prop:f_leq_g_g_cont} applying \ref{lemm:KVW_cerrado}.
\end{proof}

As a corollary of the previous propositions we obtain the following result, which generalizes Corollary 3 and Proposition 14 of \cite{stong1966finite}. 

\begin{prop}\label{prop:f_leq_g_son_hom}
Let $X$ be a topological space and let $Y$ be an Alexandroff space. Let $f,g\colon X\to Y$ be continuous maps such that $f(x)\leq g(x)$ for all $x\in X$. Let $A=\{x\in X:f(x)=g(x)\}$. Then $f\simeq g\ (\rel A)$.
\end{prop}

\begin{proof}
Follows easily from \ref{prop:f_leq_g_g_cont} or from \ref{prop:f_leq_g_f_cont}.
\end{proof}

\section{Beat points and bp--retracts in $\fintopzero/B$}

We begin this section introducing the definitions of beat points and bp--retracts in the category $\fintopzero/B$ of objects over $B$, for some $B\in\fintopzero$, and extending some results of the classical theory of Stong and of the theory of bp--retracts of \cite{cianci2019combinatorial} to this category.

\begin{definition}
Let $E$ and $B$ be finite T$_0$--spaces and let $p\colon E\to B$ be a continuous map. Let $e\in E$.
\begin{itemize}
\item We say that $e$ is a \emph{down beat point} of $p$ if it is a down beat point of both $E$ and $p^{-1}(p(e))$.
\item We say that $e$ is an \emph{up beat point} of $p$ if it is an up beat point of both $E$ and $p^{-1}(p(e))$.
\item We say that $e$ is a \emph{beat point} of $p$ if it is either a down beat point of $p$ or an up beat point of $p$.
\end{itemize}
If $e$ is a beat point of $p$, we will also say that the restriction $p|\colon E-\{e\}\to B$ of $p$ can be obtained from $p$ by removing the beat point $e$.

We will say that the map $p$ is \emph{minimal} if does not have beat points.
\end{definition}

\begin{definition}
Let $E$ and $B$ be finite T$_0$--spaces and let $p\colon E\to B$ be a continuous map. 

We will say that a map $p'\colon E'\to B$ is a \emph{dbp--retract} (resp. \emph{ubp--retract}) of $p$ if it can be obtained by successively removing down beat points (resp. up beat points) of $p$.

We will say that a continuous map $p_C\colon E'\to B$ is a \emph{core} of $p$ if it is a minimal map and a strong deformation retract of the map $p$ (cf. definition \ref{defi:rdf_sobre_B}).
\end{definition}

\begin{rem}\label{rema:bp_p}
Let $E$ and $B$ be finite T$_0$--spaces, let $p\colon E\to B$ be a continuous map and let $e_0\in E$. From the previous definition it follows that $e_0$ is a down beat point of $p$ if and only if $e_0$ is a down beat point of $E$ and $p(e_0)=p(e_1)$ where $e_1=\max(\widehat{U}_{e_0})$. In a similar way, $e_0$ is an up beat point of $p$ if and only if $e_0$ is an up beat point of $E$ and $p(e_0)=p(e_2)$ where $e_2=\min(\widehat{F}_{e_0})$.
\end{rem}

\begin{prop}\label{prop:r_flecha_sobre_B_sii_e_bp_de_p}
Let $E$ and $B$ be finite T$_0$--spaces, let $p\colon E\to B$ be a continuous map, let $e$ be a down beat point (resp. up beat point) of $E$, let $i\colon E-\{e\}\to E$ be the inclusion map and let $r\colon E\to E-\{e\}$ be the retraction associated to the removal of the down beat point (resp. up beat point) $e$.

If $e$ is a down beat point (resp. up beat point) of $p$, then $r$ is an arrow over $B$ and $pi$ is a strong deformation retract of $p$. Conversely, if $r$ is an arrow over $B$ from $p$ to $pi$, then $e$ is a down beat point (resp. up beat point) of $p$.  
\end{prop}

\begin{proof}
We will prove the case in which $e$ is a down beat point of $E$. The other case is similar.

First suppose that $e$ is a down beat point of $p$. It is clear that $i$ is an arrow over $B$ and from remark \ref{rema:bp_p} it follows that $r$ is also an arrow over $B$.
Moreover, $ri=\id_{E-\{e\}}$ and the canonical homotopy $H\colon ir\simeq \id_E$ is clearly a homotopy over $B$ relative to $E-\{e\}$. It follows that $pi$ is a strong deformation retract of $p$.

Now suppose that $r$ is an arrow over $B$ and let $e'=\max (\widehat{U}^{E}_e)$. Then
\[p(e')=p(r(e))=pir(e)=p(e).\]
The result follows from remark \ref{rema:bp_p}.
\end{proof}

The following corollary is immediate.

\begin{coro}
Every continuous map between finite T$_0$--spaces has a core.
\end{coro}

\begin{ex}
Let $X$ be a finite T$_0$--space and let $X_C$ be a core of $X$. Let $i\colon X_C\to X$ be the inclusion map and let $r\colon X\to X_C$ be a retraction of $i$ obtained by successively removing the beat points $x_0,x_1,\ldots,x_n$. We will prove that $\id_{X_C}$ is a core of $r$.

We define $X_0=X$ and, inductively, $X_{k+1}=X_k-\{x_k\}$ for $k=0,\ldots,n$. In addition, we consider for $k=0,\ldots,n$, the retraction $r_k\colon X_k\to X_{k+1}$ corresponding to the removal of the beat point $x_k$ and the inclusion map $j_k\colon X_{k+1}\to X_k$. We will prove that, for all $k=0,\ldots,n$, the point $x_k$ is a beat point of the restriction 
\[rj_0j_1\ldots j_{k-1}\colon X_{k}\to X_C\]
of $r$.

Let $k\in\{0,\ldots,n\}$. Since $r=r_n r_{n-1}\ldots r_1 r_0$, it follows that
\begin{align*}
rj_0\ldots j_k r_k&=r_n r_{n-1}\ldots r_k \ldots r_1 r_0 j_0 j_1\ldots j_k r_k=r_n r_{n-1}\ldots r_k=\\&=r_n r_{n-1}\ldots r_0 j_0\ldots j_{k-1}=r j_0\ldots j_{k-1}.
\end{align*}
Thus, by \ref{prop:r_flecha_sobre_B_sii_e_bp_de_p} it follows that $x_k$ is a beat point of the restriction $rj_0\ldots j_{k-1}$ of $r$. Hence, $r j_0 \ldots j_k$ is strong deformation retract of $r j_0\ldots j_{k-1}$.

Inductively, it is easy to see that $r j_0\ldots j_k$ is strong deformation retract of $r$ for $k=0,\ldots,n$. In particular, $\id_{X_C}=rj_0\cdots j_n$ is strong deformation retract of $r$. Moreover, since the fibers of $\id_{X_C}$ are one-point spaces, it is clear that the map $\id_{X_C}$ is minimal. It follows that $\id_{X_C}$ is a core of $r$.
\end{ex}

\begin{prop}\label{prop:sin_dbp_comparable_con_id_implica_id}
Let $E$ and $B$ be finite T$_0$--spaces and let $p\colon E\to B$ be a continuous map. If $p$ does not have down beat points (resp. up beat points) and $h\colon p\to p$ is an arrow over $B$ such that $h\leq \id_p$ (resp. $h\geq \id_p$), then $h=\id_p$.
\end{prop}

\begin{proof}
We will prove the case in which $p$ does not have down beat points and $h\leq \id_p$. The other case is similar.
Let $A=\{e\in E:h(e)<e\}$ and suppose that $A\neq\varnothing$. Let $e_0\in \mnl A$. It is not difficult to prove that $h(e_0)=\max \widehat{U}^{E}_{e_0}$ and hence $e_0$ is a down beat point of $E$. Moreover, as $h$ is an arrow over $B$, 
$ph(e_0)=p(e_0)$. From remark \ref{rema:bp_p} it follows that $e_0$ is a down beat point of $p$, which entails a contradiction.
Hence, $A$ is empty and thus $h=\id_p$.
\end{proof}

\begin{coro}\label{coro:p_mnl_entonces_h_es_id}
Let $E$ and $B$ be finite T$_0$--spaces, let $p\colon E\to B$ be a minimal map and let $h\colon p\to p$ be an arrow over $B$ such that $h\simeq \id_p$ over $B$. Then $h=\id_p$.
\end{coro}

\begin{proof}
Follows easily from \ref{prop:fence} and the previous proposition.
\end{proof}

\begin{coro}
Let $E$ and $B$ be finite T$_0$--spaces and let $p\colon E\to B$ be a continuous map. Then, two cores of $p$ are isomorphic.
\end{coro}

The results of \cite{cianci2019combinatorial} regarding bp--retracts in $\fintopzero$ can be extended in a natural way to the category $\fintopzero/B$ for all $B\in\fintopzero$. In what follows we will show how some of these results can be extended.

The following theorem is a generalization of \ref{theo:dbpr_equivalences} (\cite[Theorem 3.5]{cianci2019combinatorial}).

\begin{theo} \label{theo:dbpr_equivalences_sobre_B}
Let $A$, $X$ and $B$ be finite T$_0$--spaces with $A\subseteq X$. Let $i\colon A\to X$ be the inclusion map and let $a\colon A\to B$ and $x\colon X\to B$ be continuous maps such that $xi=a$. We consider $(A,a)$ and $(X,x)$ as objects over $B$ and the map $i\colon a\to x$ as an arrow over $B$. Then, the following propositions are equivalent:
\begin{enumerate}
\item $a$ is a dbp--retract of $x$.
\item There exists an arrow $f\colon x\to x$ over $B$ such that $f\leq\id_X$, $f^{2}=f$ and $f(X)=A$.
\item There exists a unique arrow $f\colon x\to x$ over $B$ such that $f\leq\id_X$, $f^{2}=f$ and $f(X)=A$.
\item There exists a retraction $r\colon x\to a$ of $i$ such that $ir\leq\id_X$.
\item There exists a unique retraction $r\colon x\to a$ of $i$ such that $ir\leq\id_X$.
\end{enumerate}
\end{theo}

\begin{proof}
From the definition of down beat point of $x$, it is clear that the retraction $r$ associated to the removal of a down beat point of $x$ is an arrow over $B$. Thus, the implication $(1)\Rightarrow (4)$ follows. On the other hand, the implication $(4)\Rightarrow (2)$ can be proved easily taking $f=ir$ as in the proof of Theorem 3.5 of \cite{cianci2019combinatorial}.

Now, we will prove $(2)\Rightarrow (1)$. We proceed as in the proof of Theorem 3.5 of \cite{cianci2019combinatorial}. Suppose that there exists an arrow $f\colon x\to x$ over $B$ such that $f\leq\id_X$, $f^{2}=f$ and $f(X)=A$. Let $W=\{z\in X:f(z)<z\}$. Suppose that $W\neq \varnothing$ and take $z_0\in\mnl W$. In the proof of Theorem 3.5 of \cite{cianci2019combinatorial} it is proved that $z_0$ is down beat point of $X$ with $\max (\widehat{U}_{z_0})=f(z_0)$. Since $xf=x$, it follows that $xf(z_0)=x(z_0)$ and from remark \ref{rema:bp_p}, we obtain that $z_0$ is down beat point of $x$. The proof follows as in \cite{cianci2019combinatorial}.

The rest of the implications can be proved as in \cite{cianci2019combinatorial} applying the fact that the maps involved are arrows over $B$.
\end{proof}

\begin{prop}\label{prop:dbpr_of_Y_and_sub_of_X_is_dbpr_of_X_sobre_B}
Let $A$, $B$, $X$ and $Y$ be finite T$_0$--spaces with $A\subseteq X\subseteq Y$ and let $y\colon Y\to B$ be a continuous map. Let $x\colon X\to B$ and $a\colon A\to B$ be restrictions of $y$. If $a$ is a dbp--retract of $y$ then $a$ is a dbp--retract of $x$.
\end{prop}

\begin{proof}
The proof of this proposition is analogous to the proof of Proposition 3.8 of \cite{cianci2019combinatorial}.
\end{proof}

The following proposition will be useful in the following sections.

\begin{prop}
Let $E$ and $B$ be finite T$_0$--spaces and let $p\colon E\to B$ be a continuous map. Let $\Omega$ be the set of dbp--retracts of $p$.
We define the following partial order in $\Omega$: given $x,y\in\Omega$, we define $x\leq y$ if and only if $x$ is a dbp--retract of $y$.

Then $\Omega$ has a minimum element.
\end{prop}

\begin{proof}
The proof of this proposition is similar to the proof of proposition 3.18 of \cite{cianci2019combinatorial}, applying the analogous versions of the results given in that article. 
\end{proof}

Now we will prove that the core of a fiber bundle between finite T$_0$--spaces is again a fiber bundle.

\begin{prop}\label{prop:core_fibrado}
Let $E$, $B$ and $F$ be finite T$_0$--spaces and let $p\colon E\to B$ be a fiber bundle with fiber $F$. Then the core of $p$ is a fiber bundle with base $B$ and whose fiber is a core of $F$. 
\end{prop}

\begin{proof}
By Theorem \ref{theo:classification_fiber_bundles_Grothendieck_construction}, we may assume without loss of generality that $E=B\ltimes_{D}F$ and that $p=\pi_B^{D}\colon B\ltimes_{D}F\to B$ for some functor $D\colon B\to\aut(F)$.

Let $A$ be the smallest dbp--retract of $F$. Let $i\colon A\to F$ be the inclusion map and let $r\colon F\to A$ be the retraction associated to $A$. Let $\phi$ be an automorphism of $F$. Then
\[r\phi^{-1}ir\phi i\leq r\phi^{-1}\phi i=ri=\id_{A}\]
and since $A$ does not have down beat points, it follows that $r\phi^{-1}ir\phi i=\id_{A}$. Interchanging $\phi$ and $\phi^{-1}$ we obtain that $r\phi ir\phi^{-1} i=\id_{A}$ and hence, $r\phi i$ is an automorphism of $A$ with inverse $r \phi^{-1} i$.

Let $D'\colon B\to \aut(A)$ be the functor defined by
\[D'(b\leq b')=rD(b\leq b')i\] 
for every $b,b'\in B$ such that $b\leq b'$. We will prove that $D'$ is indeed a functor. First, it is clear that 
\[D'(b\leq b)=rD(b\leq b)i=r\id_{F}i=ri=\id_{A}\]
for all $b\in B$. Now, given $b,b',b''\in B$ such that $b\leq b'\leq b''$, it follows that 
\[rD(b\leq b'')i=rD(b'\leq b'')D(b\leq b')i\geq rD(b'\leq b'')irD(b\leq b')i\] 
from where we obtain that
\[rD(b\leq b'')i=rD(b'\leq b'')irD(b\leq b')i\] 
by Lemma 8.1.1 of \cite{barmak2011algebraic}. Thus, $D'$ is a functor. 

Now, from \ref{prop:int_D_is_fiber_bundle} we obtain that that $\pi^{D'}_B\colon B\ltimes_{D'}A\to B$ is a fiber bundle over $B$ with fiber $A$. We will prove that $\pi_B^{D'}$ is the smallest dbp--retract of $\pi_B^{D}$.

Let $\phi$ be again an automorphism of $F$ and let $f=\phi i r \phi^{-1} ir$. Then
\[f^{2}=\phi i r \phi^{-1} ir\phi i r \phi^{-1} ir=\phi i (r \phi i)^{-1}(r\phi i) r \phi^{-1} ir=\phi i r \phi^{-1} ir=f.\]
In addition, it is clear that $f\leq ir$. It follows from \ref{theo:dbpr_equivalences} and \ref{coro:subseteq_iff_f1_leq_f2} that the image of $f$ is a dbp--retract of $F$ which is contained in $A$. Since $A$ is the smallest dbp--retract of $F$, it follows from \ref{coro:subseteq_iff_f1_leq_f2} that $f=ir$. Hence, $\phi i (r \phi i)^{-1} r=f=ir$ and since $r$ is an epimorphism, we obtain that $\phi i (r \phi i)^{-1}=i$, or equivalently, $\phi i=i r \phi i$.

In a similar way, taking $g=ir\phi ir\phi^{-1}$ and noting that $g^{2}=g$ and that $g\leq ir$, it can be proved that $g=ir$ and hence, $r\phi=r\phi ir$. This shows that $i$ and $r$ induce natural transformations $\iota\colon j'D'\Rightarrow jD$ and $\rho\colon jD\Rightarrow j'D'$, respectively, where $j\colon \aut(F)\hookrightarrow \Top$ and $j'\colon \aut(A)\hookrightarrow \Top$ are the inclusion functors. 

By \ref{rema:int_is_functor}, $i$ and $r$ induce morphisms of fiber bundles over $B$, $i_*\colon \pi_B^{D'}\to\pi_B^{D}$ and $r_*\colon \pi_B^{D}\to\pi_B^{D'}$. It is clear that the map $i_*\colon B\ltimes_{D'}A\to B\ltimes_{D}F$ is a inclusion map of sets. By functoriality of the topological Grothendieck construction, and since $\rho\iota=\id_{j'D'}$, it follows that $r_*i_*=\id_{B\ltimes_{D'}A}$. Hence, the inclusion map $i_*$ of $B\ltimes_{D'}A$ in $B\ltimes_{D}F$ is a subspace map, and hence, we may consider $B\ltimes_{D'}A$ as a subspace of $B\ltimes_{D}F$.

On the other hand, since $(\iota\rho)_b=ir\leq \id_F$ for all $b\in B$, an explicit calculation shows that 
\[i_*r_*(b,x)=(b,ir(x))\leq (b,x)\] 
for all $b\in B$ and all $x\in F$, from where we obtain that $i_*r_*\leq \id_{B\ltimes_{D}F}$. It follows that $B\ltimes_{D'}A$ is a dbp--retract of $B\ltimes_{D}F$. Moreover, since $i_*$ and $r_*$ are arrows over $B$, it follows from  \ref{theo:dbpr_equivalences_sobre_B} that $\pi_B^{D'}$ is a dbp--retract of $\pi_B^{D}$, and since the fibers of $\pi_B^{D'}$ do not have down beat points, we obtain that $\pi_B^{D'}$ is the smallest dbp--retract of $\pi_B^{D}$, as desired.

Thus, we have proved that the smallest dbp--retract of a fiber bundle between finite T$_0$--spaces is a fiber bundle whose fiber is the smallest dbp--retract of the fibers of the first. An analogous result for ubp--retracts of fiber bundles can be proved in a similar way. By an inductive argument, we obtain that any core of a fiber bundle between finite T$_0$--spaces is again a fiber bundle, whose fibers are homeomorphic to the cores of the fibers of the original fiber bundle.
\end{proof}

The following proposition is easy and will be applied to prove proposition \ref{prop:dbp_p_o_dbp_fibra} which shows a relationship between the beat points of an open or closed map with the beat points of its fibers and of its codomain.

\begin{prop}\label{prop:abta_cerr}
Let $E$ and $B$ Alexandroff spaces and let $p\colon E\to B$ be a continuous map. The following propositions are equivalent.
\begin{enumerate}
\item The map $p$ is open.
\item For all $e\in E$, $p(U_e)=U_{p(e)}$.
\item For all $e\in E$ and all $b\leq p(e)$, $U_e\cap p^{-1}(b)\neq\varnothing$.
\end{enumerate}
In a similar way, the following propositions are equivalent.
\begin{enumerate}[resume]
\item The map $p$ is closed.
\item For all $e\in E$, $p(F_e)=F_{p(e)}$.
\item For all $e\in E$ and all $b\geq p(e)$, $F_e\cap p^{-1}(b)\neq\varnothing$.
\end{enumerate}
\end{prop}

\begin{proof}
It is clear that $(2)\Rightarrow (1)$. We will prove that $(1)\Rightarrow (3)$. Let $e\in E$ and let $b\leq p(e)$. Since $p$ is open, $p(U_e)$ is an open subset of $B$ that contains $p(e)$ and hence $b\in U_{p(e)}\subseteq p(U_e)$. It follows that there exists $e'\in U_e$ such that $p(e')=b$. It is clear then that $e'\in U_e\cap p^{-1}(b)$.

Now we will prove $(3)\Rightarrow (2)$. Let $e\in E$. Since $p$ is continuous we have that $p(U_e)\subseteq U_{p(e)}$. Now, let $b\in U_{p(e)}$ and let $e'\in U_e \cap p^{-1}(b)$. Then $b=p(e')\in p(U_e)$ and the result follows.

The equivalence $(4)\Leftrightarrow (5)\Leftrightarrow (6)$ can be proved in a similar way.
\end{proof}

\begin{prop}\label{prop:dbp_p_o_dbp_fibra}
Let $E$ and $B$ be finite T$_0$--spaces and let $p\colon E\to B$ be a continuous map. If $p$ is an open map and $e_0$ is a down beat point of $E$, then either $p(e_0)$ is down beat point of $B$ or $e_0$ is down beat point of $p$.

In a similar way, if $p$ is a closed map and $e_0$ is an up beat point of $E$, then either $p(e_0)$ is an up beat point of $B$ or $e_0$ is an up beat point of $p$.
\end{prop}

\begin{proof}
Suppose that $p$ is an open map and that $e_0$ is a down beat point of $E$. Let $e_1=\max(\widehat{U}_{e_0})$. If $p(e_0)=p(e_1)$ the result follows from remark \ref{rema:bp_p}. Thus suppose that $p(e_1)<p(e_0)$. We will prove that $p(e_1)=\max(\hat{U}_{p(e_0)})$. Let $b<p(e_0)$. By \ref{prop:abta_cerr} there exists $e_2\in U_{e_0}\cap p^{-1}(b)$. Since $p(e_2)=b<p(e_0)$ we obtain that $e_2\neq e_0$. Hence, $e_2<e_0$ and thus $e_2\leq e_1$. Then $b=p(e_2)\leq p(e_1)$ and the result follows.

The second part of the proposition can be proved in a similar way.
\end{proof}

\section{Fibrations and beat points}

In the following, all maps will be considered non-empty.

\begin{prop}\label{prop:fib_sacar_bp}
Let $E$ and $B$ be two finite T$_0$--spaces, let $p\colon E\to B$ be a fibration and let $e$ be a beat point of $p$. Then, the restriction $p|\colon E-\{e\}\to B$ of $p$ is a fibration which is fiber homotopy equivalent to $p$.
\end{prop}

\begin{proof}
Since $p|$ is retract of $p$ we obtain that $p|$ is fibration. The result then follows by proposition \ref{prop:r_flecha_sobre_B_sii_e_bp_de_p}.
\end{proof}

\begin{theo}\label{theo:fib_al_sacar_dbp}
Let $E$ and $B$ be two finite T$_0$--spaces, let $p\colon E\to B$ be a continuous map and let $e$ be a down beat point of $p$ such that the restriction $p|\colon E-\{e\}\to B$ of $p$ is a fibration. Then $p$ is a fibration.
\end{theo}

\begin{proof}
Let $\Lambda\colon (E-\{e\})\times_{p|} B^{I}\to (E-\{e\})^{I}$ be a path-lifting map for $p|$. Let $\lambda=\Lambda^{\flat}\colon (E-\{e\})\times_{p|} B^{I}\times I\to E-\{e\}$ be the map induced by $\Lambda$ and the exponential law. Let $i\colon E-\{e\}\to E$ be the inclusion map and let $r\colon E\to E-\{e\}$ be the retraction associated to the removal of the beat point $e$. It is easy to see that the arrows $r$, $\id_B$ and $\id_{B^{I}}$ induce a morphism from the diagram $E\xrightarrow{p}B\xleftarrow{\ev_0}B^{I}$ to the diagram  $E-\{e\}\xrightarrow{p|}B\xleftarrow{\ev_0}B^{I}$ and hence, a continuous map $R\colon E\times_p B^{I}\to (E-\{e\})\times_{p|}B^{I}$ that sends the pair $(e,\gamma)\in E\times_p B^{I}$ to $(r(e),\gamma)$. The map $i\lambda\circ(R\times \id_I)\colon E\times_p B^{I}\times I\to E$ is clearly continuous and sends the element $(e,\gamma,t)$ of $E\times_p B^{I}\times I$ to $\lambda(r(e),\gamma,t)$. We define $\lambda'\colon E\times_p B^{I}\times I\to E$ by 
\[
\lambda'(e,\gamma,t)=\left\{
\begin{array}{ll}
e&\text{if $t=0$,}\\
\lambda(r(e),\gamma,t)&\text{if $t>0$.}
\end{array}
\right.
\]
Observe that $\lambda'\geq i\lambda\circ (R\times \id_I)$ and that $\lambda'$ coincides with $i\lambda\circ (R\times \id_I)$ on $E\times_p B^{I}\times (0,1]$. On the other hand, the restriction of $\lambda'$ to $E\times_p B^{I}\times \{0\}$ is the projection to $E$ and hence it is a continuous map. By \ref{prop:f_leq_g_g_cont}, it follows that $\lambda'$ is a continuous map. It is easy to verify that $\lambda'$ induces a path-lifting map $\Lambda'\colon E\times_p B^{I}\to E^{I}$ for $p$.
\end{proof}

We will see in subsection \ref{subs:fib_no_cerrada} that the previous proposition does not hold if we change down beat points to up beat points.

\begin{coro}\label{coro:dbpr_minimo_fibracion}
Let $E$ and $B$ be two finite T$_0$--spaces and let $p\colon E\to B$ be a continuous map.
The following propositions are equivalent.
\begin{enumerate}
\item $p$ is a fibration.
\item Every dbp--retract of $p$ is fibration.
\item There exists a dbp--retract of $p$ which is fibration.
\end{enumerate}
\end{coro}

\begin{proof}
The implication $(1)\Rightarrow (2)$ follows from \ref{prop:fib_sacar_bp} and the implication $(2)\Rightarrow (3)$ is immediate. The implication $(3)\Rightarrow(1)$ follows easily from \ref{theo:fib_al_sacar_dbp} applying an inductive argument.
\end{proof}

The following result generalizes \ref{prop_stong}.

\begin{theo}\label{theo:stong_para_fib}
Let $E$ and $B$ be finite T$_0$--spaces, let $p\colon E\to B$ be a fibration and let $q\colon E\to B$ be a continuous map which is homotopic to $p$. If $E$ is a minimal map, then $q=p$.
\end{theo}

\begin{proof}
Let $H\colon p\simeq q\colon E\times I\to B$ be a homotopy. Then $Hi_0=p=p\id_E$ and hence there exists $\tilde{H}\colon E\times I\to E$ such that $\tilde{H}i_0=\id_E$ and $p\tilde{H}=H$. Since $E$ is minimal and $\tilde{H}i_1\simeq \tilde{H}i_0=\id_E$, we obtain that $\tilde{H}i_1=\id_E$ by \ref{prop_stong}. Then $q=Hi_1=p\tilde{H}i_1=p\id_E=p$.
\end{proof}

\begin{coro}\label{coro:E_mnl_entonces_B_mnl}
Let $E$ and $B$ be finite T$_0$--spaces and let $p\colon E\to B$ be a fibration. If $E$ is minimal and $B$ is connected, then $B$ is minimal.
\end{coro}

\begin{proof}
We will prove that the unique continuous map $f\colon B\to B$ which is homotopic to $\id_B$ is $\id_B$. Let $f\colon B\to B$ be a continuous map such that $f\simeq \id_B$. Then $p\simeq fp$, and since $E$ is minimal, we deduce that $p=fp$ by \ref{theo:stong_para_fib}. And since (non-empty) Hurewicz fibrations over path-connected spaces are surjective it follows that $f=\id_B$ as desired.

Now let $B'$ be a bp--retract of $B$. Applying \ref{theo:dbpr_equivalences} (or its dual version for up beat points), we obtain that there exists a continuous map $f$ which is homotopic to $\id_B$ and which satisfies that $f(B)=B'$. By the result proved in the previous paragraph we obtain that $f=\id_B$, from where it follows that $B'=B$. Therefore $B$ is minimal.
\end{proof}

\begin{theo}\label{theo:Ue_cap_fibra_b}
Let $E$ and $B$ be finite T$_0$--spaces and let $p\colon E\to B$ be a fibration. Let $e\in E$ and let $b\in U_{p(e)}$. Then $U_e\cap p^{-1}(b)\neq\varnothing$.
\end{theo}

\begin{proof}
If $p(e)=b$ the result follows, since, in that case $e\in U_e\cap p^{-1}(p(e))$. Thus suppose that $b<p(e)$. Let $\S=\{0,1\}$ be the Sierpinski space with $0<1$ and let $H\colon \S\times I\to B$ the map defined by
\[
H(s,t)=\left\{
\begin{array}{ll}
p(e)&\text{if $t=0$ or $s=1$,}\\
b&\text{in other case.}
\end{array}
\right.
\]
It is clear that $H^{-1}(U_b)=\{0\}\times (0,1]$ and hence, $H$ is a continuous map. Let $i_0\colon \S\to\S\times I$ be defined by $i_0(x)=(x,0)$. Then $Hi_0=C_{p(e)}=pC_e$ and hence there exists a continuous map $\tilde{H}\colon \S\times I\to E$ such that $\tilde{H}i_0=C_e$ and $p\tilde{H}=H$. Since $\tilde{H}(1,0)=e$, we obtain that $(1,0)\in \tilde{H}^{-1}(U_e)$ and hence, there exists $\varepsilon>0$ such that $\S\times [0,\varepsilon]\subseteq \tilde{H}^{-1}(U_e)$. In particular, $\tilde{H}(0,\varepsilon)\in U_e$. But $p\tilde{H}(0,\varepsilon)=H(0,\varepsilon)=b$ and then $\tilde{H}(0,\varepsilon)\in p^{-1}(b)$. Hence $\tilde{H}(0,\varepsilon)\in U_e\cap p^{-1}(b)$.
\end{proof}

\begin{coro}\label{coro:p_abta}
Let $E$ and $B$ be finite T$_0$--spaces and let $p\colon E\to B$ be a fibration. Then $p$ is an open map.

If, in addition, $B$ is connected, then $p$ is a quotient map.
\end{coro}

\begin{proof}
Follows from \ref{theo:Ue_cap_fibra_b} and \ref{prop:abta_cerr}.
\end{proof}

In \cite{cianci2019combinatorial} we gave the following definition.

\begin{definition}[{\cite[Definition 3.15]{cianci2019combinatorial}}]
Let $X$ be a finite T$_0$--space and let $f\colon X\to X$ be a continuous function such that $f\leq\id_X$. We define $f^{\infty}$ by $f^{N}$ where $N\in\N$ is such that $f^{N}=f^{N+1}$. 
\end{definition}

It is easy to verify that, under the assumptions of the previous definition, the map $f^{\infty}$ is well defined \cite[p.240]{cianci2019combinatorial}.

The following lemma will be applied to prove \ref{theo:E_d_dbpr_de_p-1(B_d)}.

\begin{lemma}\label{lemm:dbpr_minimo_U_id_X}
Let $X$ be a finite T$_0$--space, let $A$ be the smallest dbp--retract of $X$ and let $f\colon X\to X$ be the map associated to the dbp--retract $A$ given by \ref{theo:dbpr_equivalences}, that is, the unique continuous map from $X$ to itself such that $f\leq \id_X$, $f^{2}=f$ and $f(X)=A$. Then $f$ is the minimum element of $U_{\id_X}$.
\end{lemma}

\begin{proof}
Let $g\in U_{\id_X}$. By Lemma 3.16 of \cite{cianci2019combinatorial}, $g^{\infty}\leq\id_X$ and $g^{\infty}\circ g^{\infty}=g^{\infty}$. 
By \ref{theo:dbpr_equivalences}, $g^{\infty}(X)$ is a dbp--retract of $X$ and thus $f(X)=A\subseteq g^{\infty}(X)$. By  \ref{coro:subseteq_iff_f1_leq_f2}, $f\leq g^{\infty}\leq g$. The result follows.
\end{proof}

\begin{theo}\label{theo:E_d_dbpr_de_p-1(B_d)}
Let $E$ and $B$ be finite T$_0$--spaces such that $B$ is connected. Let $p\colon E\to B$ be a fibration and let $E_d$ and $B_d$ the smallest dbp--retracts of $E$ and $B$ respectively. Then $E_d$ is a dbp--retract of $p^{-1}(B_d)$.
\end{theo}

\begin{proof}
Let $g\colon E\to E$ and $h\colon B\to B$ be the continuous maps which are associated to the dbp--retracts $E_d$ and $B_d$ respectively, given by \ref{theo:dbpr_equivalences}. By \ref{lemm:dbpr_minimo_U_id_X}, $g$ and $h$ are the minimum elements of $U_{\id_E}$ and $U_{\id_B}$ respectively. 

Since $E$ is a finite space, $p^{E}\colon E^{E}\to B^{E}$ is a Hurewicz fibration between finite T$_0$--spaces. It follows from \ref{coro:p_abta} that $p^{E}$ is an open map. In particular, $p^{E}$ sends each minimal element of $E^{E}$ to a minimal element of $B^{E}$. Since $g$ is a minimal element of $E^{E}$, $pg=p^{E}(g)$ is a minimal element of $B^{E}$. On the other hand, as $h\leq\id_B$, we have that $hpg\leq pg$ from where we obtain that $hpg=pg$. Hence $p(E_d)=p(g(E))=h(p(g(E))\subseteq \im h=B_d$ and then $E_d\subseteq p^{-1}(B_d)$. By \cite[Proposition 2.8]{cianci2019combinatorial}, we obtain that $E_d$ is a dbp--retract of $p^{-1}(B_d)$.
\end{proof}

\begin{coro}\label{coro:g=ph}
Let $E$, $B$ and $X$ be finite T$_0$--spaces and let $p\colon E\to B$ be a fibration. Let $f\colon X\to E$ be a continuous map and let $g\colon X\to B$ be a continuous map such that $g\leq pf$. Then there exists a continuous map $h\colon X\to E$ such that $h\leq f$ and $g=ph$.  
\end{coro}

\begin{proof}
Since $X$ is a finite space, $p^{X}\colon E^{X}\to B^{X}$ is a fibration between finite T$_0$--spaces. Since $g\leq pf=p^{X}(f)$ by \ref{theo:Ue_cap_fibra_b} we obtain that there exists $h\in U_f\cap (p^{X})^{-1}(g)$. Then $h\leq f$ and $ph=g$. 
\end{proof}

In \ref{coro:E_mnl_entonces_B_mnl} we proved that if $p\colon E\to B$ is a fibration between finite T$_0$--spaces and $B$ is connected and $E$ does not have beat points, then $B$ does not have beat points either. The following result shows that is also true if we consider only down beat points. We will see in \ref{subs:fib_no_cerrada}, that this result does not hold if we only consider up beat points.

\begin{coro}\label{coro:E_sin_dbp_entonces_B_sin_dbp}
Let $E$ and $B$ be finite T$_0$--spaces and let $p\colon E\to B$ be a fibration. Suppose that $E$ does not have down beat points. Then $p$ is a minimal element of $B^{E}$.

If, in addition, $B$ is connected, then $B$ does not have down beat points.
\end{coro}

\begin{proof} 
Since $E$ does not have down beat points, it follows from \ref{prop_stong} that $\id_E$ is a minimal element of $E^{E}$. Hence, as in the proof of \ref{theo:E_d_dbpr_de_p-1(B_d)}, we have that $p=p\id_E=p^{E}(\id_E)$ is a minimal element of $B^{E}$. 

Now suppose that $B$ is connected. Let $B_d$ be the minimal dbp--retract of $B$ and let $f\colon B\to B$ be the map associated to $B_d$ given by \ref{theo:dbpr_equivalences}. Since $f\leq \id_B$, we have as in \ref{theo:E_d_dbpr_de_p-1(B_d)} that $fp\leq p$ and hence $fp=p$. Since $B$ is connected, it follows that $p$ is surjective and hence $f=\id_B$. Then $B_d=B$ and thus $B$ does not have down beat points. 
\end{proof}

Combining \ref{prop:dbp_p_o_dbp_fibra} and \ref{coro:E_sin_dbp_entonces_B_sin_dbp} we immediately obtain the following result.

\begin{coro}
Let $E$ and $B$ be finite T$_0$--spaces such that $B$ is connected and let $p\colon E\to B$ be a fibration without down beat points. Then $E$ has down beat points if and only if $B$ has down beat points.
\end{coro}

The following result is a generalization of \ref{theo:Ue_cap_fibra_b}.

\begin{theo}
Let $E$ and $B$ be finite T$_0$--spaces and let $p\colon E\to B$ be a fibration. Let $e\in E$ and let $b\in U_{p(e)}$. Then $U_e\cap p^{-1}(b)$ is contractible.
\end{theo}

\begin{proof}
Let $X=U_e\cap p^{-1}(b)$ and let $\tilde{X}=X\cup \{e\}$. Let $i\colon \tilde{X}\to E$ be the inclusion map. Then $C_b\leq pi$ and by \ref{coro:g=ph} it follows that there exists $h\colon \tilde{X}\to E$ such that $h\leq i$ and $ph=C_b$. It is easy to see then that $\im h \subseteq X$.

Let $h|\colon X\to X$ be the restriction of $h$. Since $h\leq i$, it follows that $h|\leq \id_X$. On the other hand, $h|\leq C_{h(e)}$. Therefore $\id_X\simeq C_{h(e)}$ and thus $X$ is contractible.
\end{proof}

The following result can be considered as a notably weaker dual version of \ref{theo:Ue_cap_fibra_b}.

\begin{theo}\label{theo:Fe_cap_fibra_b}
Let $E$ and $B$ be finite T$_0$--spaces and let $p\colon E\to B$ be a fibration. Let $e\in E$ and let $b\in F_{p(e)}$. Then there exists $e'\in U_e \cap p^{-1}(p(e))$ such that $F_{e'}\cap p^{-1}(b)\neq\varnothing$.
\end{theo}

\begin{proof}
Let $Z$ be the topological space whose underlying set is $\{\alpha_n:n\in\N\}\cup \{\beta\}$ and whose topology is generated by the subbase $\left\{\{\alpha_n,\beta\}:n\in \N\right\}$. Observe that $Z$ is a locally finite T$_0$--space and that $U_\beta=\{\beta\}$ and $U_{\alpha_n}=\{\alpha_n,\beta\}$ for all $n\in\N$.

Let $U=\bigcup\limits_{n=1}^{\infty}U_{\alpha_n}\times [0,1/n)$. It is clear that $U$ is an open subset of $Z\times I$. We define the map $H\colon Z\times I\to B$ by 
\[
H(z,t)=\left\{
\begin{array}{ll}
p(e)&\text{if $(z,t)\in U$,}\\
b&\text{if $(z,t)\not\in U$.}
\end{array}
\right.
\]
It is easy to verify that $H$ is continuous. Note that $pC_e=Hi_0$, where $C_e\colon Z\to E$ is the constant map with value $e$. Hence, there exists a continuous map $\tilde{H}\colon Z\times I \to E$ such that $p\tilde{H}=H$ and $\tilde{H}i_0=C_e$. Now, $\tilde{H}(\beta,0)=\tilde{H}i_0(\beta)=C_e(\beta)=e$ and hence $\tilde{H}^{-1}(U_e)$ is an open subset of $Z\times I$ that contains $(\beta,0)$. Then there exists $\varepsilon>0$ such that $\{\beta\}\times [0,\varepsilon)\subseteq \tilde{H}^{-1}(U_e)$. 

Let $n\in \N$ be such that $1/n<\min\{\varepsilon,1\}$. Then $\tilde{H}(\beta,1/n)\leq e$ and $p\tilde{H}(\beta,1/n)=H(\beta,1/n)=p(e)$. Taking $e'=\tilde{H}(\beta,1/n)$ it follows that $e'\in U_e\cap p^{-1}(p(e))$. On the other hand, since $\alpha_n\in F_\beta=\overline{\{\beta\}}$, then $(\alpha_n,1/n)\in \overline{\left\{(\beta,1/n)\right\}}$ and hence
\[\tilde{H}(\alpha_n,1/n)\in \tilde{H}\left(\overline{\left\{(\beta,1/n)\right\}}\right)\subseteq \overline{\left\{\tilde{H}(\beta,1/n)\right\}}=F_{e'}.\]
But $p\tilde{H}(\alpha_n,1/n)=H(\alpha_n,1/n)=b$ and then $\tilde{H}(\alpha_n,1/n)\in p^{-1}(b)$. Hence $\tilde{H}(\alpha_n,1/n)\in F_{e'}\cap p^{-1}(b)$ and the result follows.
\end{proof}

The previous result shows that if $p\colon E\to B$ is a fibration between finite T$_0$--spaces, then for any $e\in E$ and for any $b\geq p(e)$, there exists a point $e'$ such that $e'\leq e$, $e'$ belongs to the same fiber than $e$ and $e'$ is smaller than some point of the fiber of $b$.
In subsection \ref{subs:fib_no_cerrada} we will give an example that shows that, in general, it is not true that $F_e\cap p^{-1}(b)\neq\varnothing$ for any $b\geq p(e)$.

\begin{coro}
Let $E$, $B$ and $X$ be finite T$_0$--spaces and let $p\colon E\to B$ be a fibration. Let $f\colon X\to E$ and $g\colon X\to B$ be continuous maps such that $g\geq pf$. Then there exist continuous maps $h,i\colon X\to E$ such that $i\geq h\leq f$, $pf=ph$ and $g=pi$.  
\end{coro}

\begin{proof}
The proof of this result is analogous to the proof of \ref{coro:g=ph}.
\end{proof}

\begin{prop}\label{prop:fin_sin_dbp_cerrada}
Let $E$ and $B$ be finite T$_0$--spaces and let $p\colon E\to B$ be a fibration without down beat points. Then $p$ is a closed map.
\end{prop}

\begin{proof}
By \ref{prop:abta_cerr}, it suffices to prove that for all $e\in E$ and all $b\geq p(e)$, there exists $e'\in F_e\cap p^{-1}(b)$. Let $e\in E$, let $b\geq p(e)$ and let $p|\colon p^{-1}(U_b)\to U_b$ be the restriction of $p$. Then $p|$ is the pullback of $p$ along the inclusion map $U_b\hookrightarrow B$ and hence it is a fibration. Since $p^{-1}(U_b)$ is an open subset of $E$, any down beat point of $p^{-1}(U_b)$ is a down beat point of $E$. Thus the down beat points of $p|$ are down beat points of $p$. In particular, $p|$ does not have down beat points.

Now, $p|\circ\id_{p^{-1}(U_b)}=p|\leq C_b$ and hence there exist continuous maps $h,i\colon p^{-1}(U_b)\to p^{-1}(U_b)$ such that $i\geq h\leq \id_{p^{-1}(U_b)}$, $p|i=C_b$ and $p|h=p|$. In particular, $h$ is an arrow over $B$ from $p|$ to $p|$ and $h\leq \id_{p^{-1}(U_b)}$. Since $p|$ does not have down beat points, it follows from \ref{prop:sin_dbp_comparable_con_id_implica_id} that $p|=\id_{p^{-1}(U_b)}$ and hence $i\geq \id_{p^{-1}(U_b)}$. Since $i(e)\geq e$ and $p(i(e))=b$, we obtain that $i(e)\in F_e\cap p^{-1}(b)$.
\end{proof}

\begin{prop}\label{prop:p_fib_cerrada_ubp_p_o_ubp_fibra}
Let $E$ and $B$ be finite T$_0$--spaces and let $p\colon E\to B$ be a Hurewicz fibration.
\begin{enumerate}
\item If $e_0$ is a down beat point of $E$, then $p(e_0)$ is a down beat point of $B$ or $e_0$ is a down beat point of $p$.
\item If $p$ does not have down beat points and $e_0$ is an up beat point of $E$, then $p(e_0)$ is an up beat point of $B$ or $e_0$ is an up beat point of $p$.
\end{enumerate}
\end{prop}

\begin{proof}
The first item follows from \ref{prop:dbp_p_o_dbp_fibra} and \ref{coro:p_abta}. The second item follows from \ref{prop:dbp_p_o_dbp_fibra} and \ref{prop:fin_sin_dbp_cerrada}.
\end{proof}

\begin{theo}
Let $E$ and $B$ be finite T$_0$--spaces such that $B$ is connected and let $p\colon E\to B$ be a fibration which is a minimal map. Then $E$ is minimal if and only if $B$ is minimal.
\end{theo}

\begin{proof}
In \ref{coro:E_mnl_entonces_B_mnl} we have proved that if $E$ is minimal then $B$ is also minimal. Thus suppose that $B$ is minimal. Since $p$ does not have down beat points, it follows from \ref{prop:p_fib_cerrada_ubp_p_o_ubp_fibra} that if $e_0$ is a beat point of $E$ then $e_0$ is a beat point of $p$ or $p(e_0)$ is a beat point of $B$. Since $p$ and $B$ do not have beat points, it is clear that $E$ will not have beat points either.
\end{proof}

A Hurewicz fibration $p\colon E\to B$ is called \emph{trivial} if it is fiber homotopy equivalent to the projection $B\times F\to B$ for some space $F$ which is homotopy equivalent to the fibers of $p$. If $F'$ is a space which is homotopy equivalent to $F$, then the projection $B\times F\to B$ is fiber homotopy equivalent to the projection $B\times F'\to B$. It follows that $p$ is fiber homotopy equivalent to the projection $B\times F'\to B$.

\begin{prop}
Let $E$ and $B$ be finite T$_0$--spaces such that $B$ is connected and let $p\colon E\to B$ be a trivial fibration which is a minimal map. Then $p$ is isomorphic to the projection $\pi_B\colon B\times F\to B$ for some minimal finite T$_0$--space $F$. In particular, the fibers of $p$ are minimal spaces.
\end{prop}

\begin{proof}
Let $F$ be the core of a fiber of $p$. Since $p$ is a trivial fibration and its fibers have the homotopy type of $F$, then $p$ is fiber homotopy equivalent to the projection $\pi_B\colon B\times F\to B$. Let $f\colon p\to \pi_B$ be a homotopy equivalence over $B$ with inverse $g$. Since $F$ is minimal, then $\pi_B$ is minimal. And since $p$ and $\pi_B$ are minimal maps, it follows from \ref{coro:p_mnl_entonces_h_es_id} that $fg=\id_{\pi_B}$ and $gf=\id_{p}$. Hence $f$ and $g$ are isomorphisms over $B$. Thus $p\cong \pi_B$. In particular, the fibers of $p$ are all homeomorphic to $F$, which is a minimal space.
\end{proof}

\section{Hurewicz and Grothendieck fibrations}

A continuous map between T$_0$ topological spaces can be viewed as a functor between posets and hence we can study whether it is a Grothendieck fibration or not (cf. definition \ref{defi:fib_grothendieck}). Since in a poset the arrows only depend on its domain and its codomain, the definition of Grothendieck fibration between posets adopts a much simpler form that is stated in the following lemma.

\begin{lemma}\label{lemm:fib_groth_entre_posets}
Let $E$ and $B$ be Alexandroff T$_0$--spaces and let $p\colon E\to B$ be a continuous map. For each $e\in E$ and each $b\leq p(e)$, the unique arrow in $B$ of $b$ to $p(e)$ has a cartesian lift to $e$ if and only if there exists $e'\in p^{-1}(b)$ such that $e'=\max (U_e\cap p^{-1}(U_b))$, in which case, the arrow $e'\to e$ is the only cartesian lift of $b\to p(e)$ to $e$.

In particular, $p$ is a Grothendieck fibration if and only if for all $e\in E$ and all $b\leq p(e)$, the set $U_e\cap p^{-1}(U_b)$ has a maximum element and this maximum element belongs to $p^{-1}(b)$.
\end{lemma}

\begin{proof}
If the cartesian lift of an arrow to a given object exists, then it is unique up to composition with a unique isomorphism. Since in a poset there do not exist non-trivial isomorphisms, then the cartesian lift of an arrow $b\to p(e)$ for $e\in E$ and $b\leq p(e)$ is unique if it exists.

Let $e\in E$ and let $b\leq p(e)$. Suppose that there exists $e'$ such that the unique arrow $e'\to e$ is a cartesian lift of $b\to p(e)$. It is clear that $e'\in p^{-1}(b)$ and that $e'\in U_e\cap p^{-1}(U_b)$. Now let $e''\in U_e\cap p^{-1}(U_b)$. Since the arrow $e'\to e$ is cartesian, there exists an arrow $e''\to e'$. It follows that $e''\leq e'$. Hence $e'=\max U_e\cap p^{-1}(U_b)$.

Now suppose that $U_e\cap p^{-1}(U_b)$ has a maximum element $e'\in p^{-1}(b)$. It is clear that the arrow $e'\to e$ is a lift of $b\to p(e)$ to $e$. On the other hand, for each arrow $e''\to e$ such that $p(e'')\leq b$, we have that $e''\in p^{-1}(U_b)$ and hence $e''\leq e'$. Therefore there exists a unique arrow $e''\to e'$ and the composition $e''\to e'\to e$ is equal to $e''\to e$. Thus, the arrow $e'\to e$ is cartesian. The result follows.
\end{proof}

\begin{rem}\label{rema:cleavage_for_opfibration}
It follows from \ref{lemm:fib_groth_entre_posets} that a continuous map $p\colon E\to B$ between Alexandroff T$_0$--spaces is a Grothendieck opfibration if and only if for all $e\in E$ and all $b\geq p(e)$, the set $F_e\cap p^{-1}(F_b)$ has a minimum element and this minimum element belongs to $p^{-1}(b)$.
\end{rem}

In \cite{grothendieck1971revetement} it is observed that if a category $E$ does not have non-trivial isomorphisms, then the unique possible cleavage for a Grothendieck fibration $p\colon E\to B$ is closed. This is also mentioned in section B1.3 of \cite{johnstone2002sketches}. In particular, $p$ is a split fibration. We will prove the particular case of this relevant fact in our context, that is, in the case that $p$ is a Grothendieck fibration between posets.

\begin{lemma}\label{lemm:fib_entre_posets_es_escindida}
Let $E$ and $B$ be Alexandroff T$_0$--spaces and let $p\colon E\to B$ be a Grothendieck fibration (resp. Grothendieck opfibration). Then there exists a unique cleavage (resp. opcleavage) for $p$ and that cleavage is closed.
\end{lemma}

\begin{proof}
Suppose that $p$ is a Grothendieck fibration. By the previous lemma, if $e\in E$ and $b\leq p(e)$ then there exists a unique cartesian lift $e'\to e$ of $b\to p(e)$ to $e$ that is the unique arrow of $e'$ to $e$, where $e'=\max (U_e\cap p^{-1}(U_b))$ and $e'\in p^{-1}(b)$. It follows that there exists a unique cleavage for $p$. We will prove that this cleavage is closed. 

It is clear that the cartesian lift of $\id_{p(e)}$ to $e$ is $\id_e$, for all $e\in E$. Now, suppose that $b'\leq b\leq p(e)$. We want to prove that the cartesian lift of $b'\to p(e)$ to $e$ coincides with the composition $e''\to e'\to e$ where $e'\to e$ is the cartesian lift of $b\to p(e)$ to $e$ and $e''\to e'$ is the cartesian lift of $b'\to b$ to $e'$. In other words, we want to prove that 
\[\max(U_e \cap p^{-1}(U_{b'}))=\max (U_{e'}\cap p^{-1}(U_{b'}))\]
where $e'=\max (U_e\cap p^{-1}(U_{b}))$.

Since $e'\leq e$ it follows that $U_{e'}\subseteq U_e$, from where we obtain that $U_{e'}\cap p^{-1}(U_{b'})\subseteq U_e\cap p^{-1}(U_{b'})$ and hence, that
\[\max(U_e \cap p^{-1}(U_{b'}))\geq\max (U_{e'}\cap p^{-1}(U_{b'})).\]
On the other hand, as $b'\leq b$, then $U_{b'}\subseteq U_b$ from where we obtain that $p^{-1}(U_{b'})\subseteq p^{-1}(U_b)$. Thus, \[\max (U_e\cap p^{-1}(U_{b'}))\leq \max(U_e\cap p^{-1}(U_b))=e'.\]
Hence,
\[\max (U_e\cap p^{-1}(U_{b'}))\in U_{e'}\cap p^{-1}(U_{b'}).\]
Thus, it is clear that 
\[\max (U_e\cap p^{-1}(U_{b'}))\leq \max( U_{e'}\cap p^{-1}(U_{b'})).\]
The result follows.

The case in which $p$ is a Grothendieck opfibration follows applying the previous case to $p^{\op}$.
\end{proof}

It is known that the split Grothendieck fibrations are those functors that can be realized as Grothendieck constructions over contravariant functors to $\cat$ (Theorem 1.3.5, B1.3 of \cite{johnstone2002sketches}, see also \cite{gray1966fibred}). In a similar way, the split Grothendieck opfibrations are those functors that can be realized as Grothendieck constructions over covariant functors to $\cat$.

Now, it is clear that if $p\colon E\to B$ is a Grothendieck opfibration between posets, then the fibers of $p$ are posets and hence $p$ can be realized as the Grothendieck construction over a covariant functor to $\pos$, or equivalently, as the topological Grothendieck construction over a covariant functor to $\Top$. Conversely, if $p$ is the projection associated to the topological Grothendieck construction over a functor from a poset $B$ in $\Top$ that maps each element of $B$ to an Alexandroff T$_0$--space, then it coincides with the classical Grothendieck construction over a functor to $\cat$ and hence it is a Grothendieck opfibration. Thus, the Grothendieck opfibrations over $B$ between posets are those functors that can be realized, in a canonical way, as projections associated to topological Grothendieck constructions of functors from $B$ to $\pos$.

In a similar way, the Grothendieck fibrations over $B$ between posets are those functors that can be realized as projections associated to topological Grothendieck constructions of functors from $B^{\op}$ to $\pos$.

\begin{theo}\label{theo:p_alpha_beta}
Let $E$ and $B$ be Alexandroff T$_0$--spaces and let $p\colon E\to B$ be a continuous map.
\begin{enumerate}
\item The map $p$ is a Grothendieck fibration if and only if there exists a functor $\alpha\colon B^{\op}\to \Top$ 
such that $\pi^{\alpha}_{B^{\op}}\cong p^{\op}$. In that case, the functor $\alpha$ can be defined canonically by
\[\alpha(b)=(p^{\op})^{-1}(b)=p^{-1}(b)^{\op}\] 
for all $b\in B$ and by
\[\alpha(b'\geq b)(e)=\max(U_e\cap p^{-1}(b))\]
for all $e\in p^{-1}(b')$ and all $b,b'\in B$ such that $b\leq b'$.
\item $p$ is a Grothendieck opfibration if and only if there exists a functor $\beta\colon B\to \Top$ such that $\pi^{\beta}_{B}\cong p$. In that case, the functor $\beta$ can be canonically defined by
\[\beta(b)=p^{-1}(b)\] 
for all $b\in B$ and by 
\[\beta(b\leq b')(e)=\min(F_e\cap p^{-1}(b'))\] 
for all $e\in p^{-1}(b)$ and all $b,b'\in B$ such that $b\leq b'$.
\end{enumerate}
\end{theo}

\begin{proof}
Item $(2)$ follows from the version of \ref{theo:pi_B_iso_p_sobre_B} for covariant functors applying the unique opcleavage for $p$ that can be obtained from \ref{rema:cleavage_for_opfibration}, which is closed by \ref{lemm:fib_entre_posets_es_escindida}.

Item $(1)$ can be proved as follows. If $p$ is a Grothendieck fibration then $p^{\op}$ is a Grothendieck opfibration and hence, by $(2)$, there exists a functor $\alpha\colon B^{\op}\to \Top$ such that $\pi^{\alpha}_{B^{\op}}\cong p^{\op}$. Conversely, if there exists a functor $\alpha$ such that $\pi^{\alpha}_{B^{\op}}\cong p^{\op}$, since $\pi^{\alpha}_{B^{\op}}$ is a Grothendieck opfibration, we have that $p^{\op}$ is also a Grothendieck opfibration. Then $p$ is a Grothendieck fibration.
\end{proof}

In the following lemma, we state some properties of the functors $\alpha$ and $\beta$ constructed in \ref{theo:p_alpha_beta}.

\begin{lemma}\label{lemm:props_de_alpha_beta}
Let $E$ and $B$ be finite T$_0$--spaces and let $p\colon E\to B$ be a Grothendieck bifibration. Let $\alpha$ and $\beta$ the functors defined in \ref{theo:p_alpha_beta} and let $b,b'\in B$ such that $b\leq b'$. Then
\begin{enumerate}
\item $\alpha(b'\geq b)\beta(b\leq b')\geq \id_{p^{-1}(b)}$.
\item $\beta(b\leq b')\alpha(b'\geq b)\beta(b\leq b')=\beta(b\leq b')$.
\item $\beta(b\leq b')\alpha(b'\geq b)\leq \id_{p^{-1}(b')}$.
\item $\alpha(b'\geq b)\beta(b\leq b')\alpha(b'\geq b)=\alpha(b'\geq b)$.
\end{enumerate}
\end{lemma}

\begin{proof}
First we will prove item $(1)$. Let $e\in p^{-1}(b)$. Then $e\leq \beta(b\leq b')(e)$ and hence, $e\in U_{\beta(b\leq b')(e)}\cap p^{-1}(b)$ from where it is clear that $\alpha(b'\geq b)\beta(b\leq b')(e)\geq e$. 

Item $(3)$ can be proved in a similar way.

From $(1)$ it follows that 
\[\alpha(b'\geq b)\beta(b\leq b')\alpha(b'\geq b)\geq \alpha(b'\geq b)\] 
and that 
\[\beta(b\leq b')\alpha(b'\geq b)\beta(b\leq b')\geq \beta(b\leq b'),\] 
while from $(3)$ it follows that 
\[\beta(b\leq b')\alpha(b'\geq b)\beta(b\leq b')\leq \beta(b\leq b')\] 
and that \[\alpha(b'\geq b)\beta(b\leq b')\alpha(b'\geq b)\leq \alpha(b'\geq b).\] 
Thus, items $(2)$ and $(4)$ follow.
\end{proof}

\begin{rem} \label{rem_alpha_beta_ubp_dbp_retracts}
Let $E$ and $B$ be finite T$_0$--spaces, let $p\colon E\to B$ be a Grothendieck bifibration and let $\alpha$ and $\beta$ be the functors defined in \ref{theo:p_alpha_beta}. Let $b,b'\in B$ be such that $b\leq b'$. From \ref{lemm:props_de_alpha_beta} we obtain that $\alpha(b'\geq b)$ and $\beta(b\leq b')$ are homotopy equivalences. 
Moreover, from \ref{lemm:props_de_alpha_beta} and \ref{theo:dbpr_equivalences} it follows that $\alpha(b'\geq b)\beta(b\leq b')(p^{-1}(b))$ is an ubp--retract of $p^{-1}(b)$ that is homeomorphic to the space $\beta(b\leq b')\alpha(b'\geq b)(p^{-1}(b'))$, which is a dbp--retract of $p^{-1}(b')$.

However, as the following example shows, in general it is not true that the minimum ubp--retract of $p^{-1}(b)$ is homeomorphic to a dbp--retract of $p^{-1}(b')$, neither is it true that the smallest dbp--retract of $p^{-1}(b')$ is homeomorphic to an ubp--retract of $p^{-1}(b)$. 

In particular, in general it is not true that the minimum ubp--retract of $p^{-1}(b)$ is homeomorphic to the smallest dbp--retract of $p^{-1}(b')$.
\end{rem}

\begin{ex}
Let $B$ be the poset with underlying set $\{a,b,c,d\}$ with partial order generated by $a<b$, $c<b$ and $c<d$ and let $\S$ be the Sierpinski space. Let $\pi_\S\colon B\times\S\to \S$ be the canonical projection, which is represented in the following diagram.

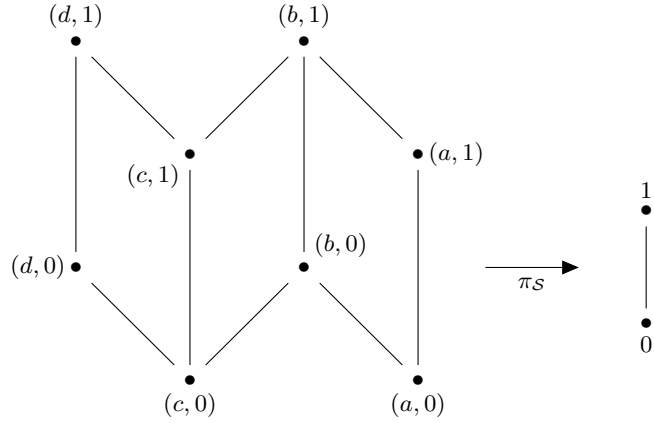
\begin{center}
\begin{tikzpicture}[x=1.5cm,y=1.5cm]
\tikzstyle{every node}=[font=\footnotesize]

\draw (0,1) node (d0){$\bullet$} node[left]{$(d,0)$};
\draw (1,0) node (c0){$\bullet$} node[below=1]{$(c,0)$};
\draw (2,1) node (b0){$\bullet$} node[above right]{$(b,0)$};
\draw (3,0) node (a0){$\bullet$} node[below=1]{$(a,0)$};

\draw (0,3) node (d1){$\bullet$} node[above=1]{$(d,1)$};
\draw (1,2) node (c1){$\bullet$} node[below left]{$(c,1)$};
\draw (2,3) node (b1){$\bullet$} node[above=1]{$(b,1)$};
\draw (3,2) node (a1){$\bullet$} node[right]{$(a,1)$};

\draw (a0)--(a1);
\draw (b0)--(b1);
\draw (c0)--(c1);
\draw (d0)--(d1);
\draw (a0)--(b0);
\draw (c0)--(b0);
\draw (c0)--(d0);
\draw (a1)--(b1);
\draw (c1)--(b1);
\draw (c1)--(d1);

\draw (3.5,1) node (s){};
\draw (4.5,1) node (t){};
\draw[\arr] (s) -> (t) node [midway,below] {$\pi_\S$};
\draw (5,0.5) node (0){$\bullet$} node [below=1]{$0$};
\draw (5,1.5) node (1){$\bullet$} node [above=1]{$1$};
\draw (0)--(1);
\end{tikzpicture}
\end{center}

It is easy to see that $\pi_\S$ is a Grothendieck bifibration, that the minimal ubp--retract of $\pi^{-1}_\S(0)$ is not homeomorphic to any dbp--retract of $\pi^{-1}_\S(1)$ and that the minimal dbp--retract of $\pi^{-1}_\S(1)$ is not homeomorphic to any ubp--retract of $\pi^{-1}_\S(0)$.
\end{ex}

The following proposition follows immediately from \ref{rem_alpha_beta_ubp_dbp_retracts}.

\begin{prop}\label{prop_alpha_y_beta_fib_hom_eq}
Let $E$ and $B$ be finite T$_0$--spaces with $B$ connected and let $p\colon E\to B$ be a Grothendieck bifibration. Then the fibers of $p$ are homotopy equivalent.
\end{prop}

\begin{theo}
Let $E$ and $B$ be finite T$_0$--spaces with $B$ connected and let $p\colon E\to B$ be a Grothendieck bifibration with minimal fibers. Then $p$ is a fiber bundle.
\end{theo}

\begin{proof}
Consider the functors $\alpha$ and $\beta$ defined in \ref{theo:p_alpha_beta}. Since the fibers of $p$ are minimal, by \ref{prop_stong} and \ref{lemm:props_de_alpha_beta}, the maps $\alpha(b'\geq b)$ and $\beta(b\leq b')$ are mutually inverse homeomorphisms for every $b,b'\in B$ such that $b\leq b'$. In particular, $\beta$ is a morphism-inverting functor. It follows from \ref{prop:int_D_is_fiber_bundle} that $\pi^{\beta}_B$ is a fiber bundle. By \ref{theo:p_alpha_beta}, we have that $p\cong \pi^{\beta}_B$ and hence $p$ is also a fiber bundle.
\end{proof}

However, it is not true that a minimal Grothendieck bifibration between finite T$_0$--spaces is a fiber bundle, as the following example shows.

\begin{ex}
Let $B$ be the finite topological space with underlying set $\{a,b,c,d\}$ and with topology generated by the basis $\{\{a\},\{b\},\{a,c,d\},\{b,c,d\}\}$. Let $C$ be the finite topological space with underlying set $\{0,1,2\}$ and with topology generated by the basis $\{\{0\},\{0,1\},\{0,2\}\}$. Let $E=B\times C$ and let $p\colon E\to B$ be the projection map.
\medskip
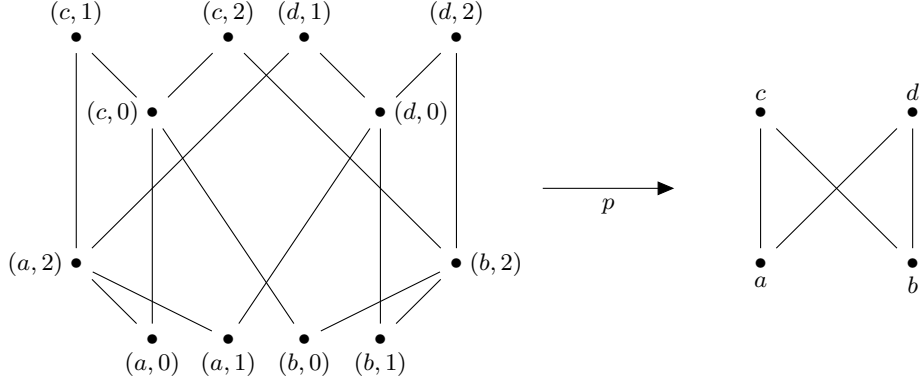
\begin{center}
\begin{tikzpicture}[x=2cm,y=2cm]
\tikzstyle{every node}=[font=\footnotesize]
\draw (0.5,0) node (a0){$\bullet$} node[below=1]{$(a,0)$};
\draw (1,0) node (a1){$\bullet$} node[below=1]{$(a,1)$};
\draw (0,0.5) node (a2){$\bullet$} node[left=1]{$(a,2)$};
\draw (1.5,0) node (b0){$\bullet$} node[below=1]{$(b,0)$};
\draw (2,0) node (b1){$\bullet$} node[below=1]{$(b,1)$};
\draw (2.5,0.5) node (b2){$\bullet$} node[right=1]{$(b,2)$};
\draw (0.5,1.5) node (c0){$\bullet$} node[left=1]{$(c,0)$};
\draw (0,2) node (c1){$\bullet$} node[above=1]{$(c,1)$};
\draw (1,2) node (c2){$\bullet$} node[above=1]{$(c,2)$};
\draw (2,1.5) node (d0){$\bullet$} node[right=1]{$(d,0)$};
\draw (1.5,2) node (d1){$\bullet$} node[above=1]{$(d,1)$};
\draw (2.5,2) node (d2){$\bullet$} node[above=1]{$(d,2)$};

\draw (a0)--(a2);
\draw (a0)--(c0);
\draw (a1)--(a2);
\draw (a1)--(d0);
\draw (b0)--(b2);
\draw (b0)--(c0);
\draw (b1)--(b2);
\draw (b1)--(d0);
\draw (c0)--(c1);
\draw (c0)--(c2);
\draw (d0)--(d1);
\draw (d0)--(d2);
\draw (a2)--(c1);
\draw (a2)--(d1);
\draw (b2)--(c2);
\draw (b2)--(d2);
\draw (3,1) node (s){};
\draw (4,1) node (t){};
\draw[\arr] (s) -> (t) node [midway,below] {$p$};
\draw (4.5,0.5) node (a){$\bullet$} node [below=1]{$a$};
\draw (4.5,1.5) node (c){$\bullet$} node [above=1]{$c$};
\draw (5.5,0.5) node (b){$\bullet$} node [below=1]{$b$};
\draw (5.5,1.5) node (d){$\bullet$} node [above=1]{$d$};
\draw (a)--(c);
\draw (a)--(d);
\draw (b)--(c);
\draw (b)--(d);
\end{tikzpicture}
\end{center}

An exhaustive verification shows that $p$ is a Grothendieck bifibration and that $p$ is a minimal map. Since the fibers of $p$ are not homeomorphic, it is clear that $p$ is not a fiber bundle. 

Observe that, by \ref{prop:core_fibrado}, $p$ can not be obtained from a fiber bundle between finite T$_0$--spaces by successively removing beat points of that fiber bundle.
\end{ex}

\begin{lemma}\label{lemm:altura_1_con_maximo}
Let $E$ and $B$ be finite T$_0$--spaces and let $p\colon E\to B$ be a Grothendieck bifibration. Suppose, in addition, that $B$ has a maximum element $b_0$ and that $h(B)=1$ (that is, $B$ is homeomorphic to the non-Hausdorff cone of a discrete finite topological space). Then $p$ is a retract of the canonical projection $\pi_B\colon E\times B\to B$.
\end{lemma}

\begin{proof}
The idea of the proof is the following. First observe that the fiber $p^{-1}(b_0)$ is an ubp--retract of $E$ and hence, we may retract $E\times B$ to $E\times \widehat{U}_{b_0}\cup p^{-1}(b_0)\times \{b_0\}$ by successively removing up beat points of the projection to $B$. After that, we may retract each subspace $E\times \{b\}$ with $b<b_0$ to $p^{-1}(F_b)\times\{b\}$, removing up beat points again. These two steps can be done simultaneously, obtaining an ubp--retract of $\pi_B$. Finally, we may retract each subspace $p^{-1}(F_b)\times\{b\}$ with $b<b_0$ to $p^{-1}(b)\times \{b\}$ removing down beat points. The space obtained is homeomorphic over $B$ to $E$, and thus it follows that $p$ is a retract of $\pi_B$. In what follows, we will state formally and prove these facts.

Let $X=\bigcup\limits_{b\in B} p^{-1}(F_b)\times \{b\}$, let $i_X\colon X\to E\times B$ be the inclusion map and let $r_X\colon E\times B\to X$ be the map defined by
\[r_X(e,b)=\left\{
\begin{array}{ll}
(e,b)&\text{if $p(e)=b$,}\\
(\min (F_e\cap p^{-1}(b_0)),b)&\text{if $p(e)\neq b$.}
\end{array}
\right.
\] 

It is easy to see that $r_X$ is well defined, that $\pi_B i_X r_X=\pi_B$, that $r_X i_X=\id_X$ and that $i_X r_X\geq \id_{E\times B}$. Hence, if we prove that $r_X$ is continuous, from the dual version for ubp--retracts of \ref{theo:dbpr_equivalences_sobre_B} it will follow that $\pi_B i_X$ is an ubp--retract of $\pi_B$. And since $i_X$ is a subspace map, it suffices to prove that $i_Xr_X$ is continuous. Moreover, since $\pi_B i_X r_X=\pi_B$, it suffices to prove that $\phi=\pi_E i_X r_X$ is continuous, where $\pi_E\colon E\times B\to E$ is the canonical projection.

Note that $\phi\geq \pi_E$. Let $U=\bigcup\limits_{b\in \mnl B}(p^{-1}(b)\times \{b\})$. Clearly $U$ is an open subset of $E\times B$ and $\phi$ coincides with $\pi_E$ in $U$. On the other hand, the restriction of $\phi$ to $U^{c}$ is the map $\phi|$ defined by
\[\phi|(e,b)=\min (F_e \cap p^{-1}(b_0))\]
for all $(e,b)\in U^{c}$. Then, it is easy to prove that $\phi|$ is continuous. From \ref{prop:f_leq_g_g_cont}, it follows that $\phi$ is continuous as desired. Thus, $X$ is an ubp--retract of $E\times B$ and $\pi_B i_X$ is an ubp--retract of $\pi_B$.

Let $i\colon E\to X$ be the map defined by $i(e)=(e,p(e))$ and let $\rho\colon X\to E$ be the map defined by $\rho(e,b)=\max (U_e\cap p^{-1}(b))$ for all $(e,b)\in X$. It is clear that $i$ is continuous, and since $\pi_E i_X i=\id_E$, it follows that $i$ is a subspace map.

We will prove now that $\rho$ is continuous. Let $(e,b),(e',b')\in X$ be such that $(e,b)\leq (e',b')$. Then $U_e\subseteq U_{e'}$ and $U_b\subseteq U_{b'}$. It follows that
\begin{align*}
\rho(e,b)&=\max(U_e\cap p^{-1}(b))=\max(U_e\cap p^{-1}(U_b))\leq\\
&\leq \max (U_{e'}\cap p^{-1}(U_{b'}))=\max(U_{e'}\cap p^{-1}(b'))=\rho(e',b')
\end{align*}
where the second and third equalities hold by \ref{lemm:fib_groth_entre_posets}. Thus, $\rho$ is continuous.

On the other hand, it is easy to see that $\pi_B i_X i=p$, that $p\rho=\pi_B i_X$, that $\rho i=\id_E$ and that $i\rho\leq \id_X$.
It follows that $E$ is a dbp--retract of $X$ and that $p=\pi_B i_X i$ is a dbp--retract of $\pi_B i_X$. In particular, $p$ is a retract of $\pi_B$ as we wanted to prove.
\end{proof}

In  order to stablish a relationship between Hurewicz fibrations between finite T$_0$--spaces and Grothendieck fibrations, we will need to study the regularity of path lifting maps associated to Hurewicz fibrations and its relationship with the existence of beat points.

\begin{theo}\label{theo:fib_sin_dbp_regular}
Let $E$ and $B$ be finite T$_0$--spaces and let $p\colon E\to B$ be a fibration. Then:
\begin{enumerate}
\item If $p$ is minimal, all path-lifting maps for $p$ are regular.
\item If $p$ does not have down beat points, then there exists a regular path-lifting map for $p$.
\end{enumerate}
\end{theo}

\begin{proof}
For each $t\in I$, let $i_t\colon E\to E\times I$ be defined by $e\mapsto (e,t)$. Let $\pi_E\colon E\times I \to E$ be the canonical projection. 

First we will prove $(1)$. Suppose that $p$ is a minimal map and that $\Lambda\colon E\times_p B^{I}\to E^{I}$ is a path-lifting map for $p$.
Then $p\pi_E i_0=p=p\id_E$ and hence we may apply $\Lambda$ to lift $p\pi_E$ from $\id_E$, obtaining a continuous map $\tilde{H}\colon E\times I\to E$ such that $\tilde{H}i_0=\id_E$ and $p\tilde{H}=p\pi_E$. It is not difficult to verify that $\widetilde{H}$ is equal to $(\Lambda\phi)^{\flat}$, the map induced by $\Lambda\phi$ by the exponential law, where $\phi\colon E\to E\times_p B^{I}$ is the map induced in the pullback by the maps $\id_E$ and $(p\pi_E)^{\sharp}$. Explicitly, $\widetilde{H}(e,t)=\Lambda(e,C_{p(e)})(t)$ for all $e\in E$ and $t\in I$. 
Hence, it suffices to prove that $\widetilde{H}(e,t)=e$ for every $e\in E$ and $t\in I$.
 
Note that, in particular, $p\tilde{H}i_t=p\pi_E i_t=p$ for all $t\in I$. Hence, $\tilde{H}$ is a homotopy over $B$.
Then $\tilde{H}i_t$ is homotopic to $\id_p$ over $B$ for all $t\in I$. Since $p$ is minimal, it follows from \ref{coro:p_mnl_entonces_h_es_id} that $\tilde{H}i_t=\id_p$. Hence $\tilde{H}(e,t)=e$ for all $e\in E$ and all $t\in I$ as desired. The result follows.

Now we will prove $(2)$. Suppose that $p$ does not have down beat points and that $\Lambda\colon E\times_p B^{I}\to E^{I}$ is a path-lifting map for $p$. As we have done previously, we use the map $\Lambda$ to lift $p\pi_E$ from $\id_E$, obtaining again a homotopy $\tilde{H}\colon E\times I\to E$ such that $\tilde{H}i_0=\id_E$ and $p\tilde{H}=p\pi_E$. For each $e\in E$, $\tilde{H}(e,0)=e\in U_e$ and hence there exists $\varepsilon_e>0$ such that $\{e\}\times [0,\varepsilon_e]\subseteq \tilde{H}^{-1}(U_e)$. Let $\varepsilon=\min\{\varepsilon_e:e\in E\}>0$. Then $\tilde{H}i_t\leq \id_E$ for all $t\in[0,\varepsilon]$. Since $p$ does not have down beat points, it follows from \ref{prop:sin_dbp_comparable_con_id_implica_id} that $\tilde{H}i_t=\id_E$ for all $t\in[0,\varepsilon]$. In particular, $\Lambda$ lifts constant paths to paths that are constant in the interval $[0,\varepsilon]$.

Given $\gamma\colon I\to B$, we define $\tilde{\gamma}\colon I\to B$ by $\tilde{\gamma}(t)=\gamma(\min\{1,t/\varepsilon\})$.
Note that the assignment  $t\mapsto \min\{1,t/\varepsilon\}$ from $I$ to $I$ is continuous. Applying the exponential law it is easy to prove that the map $\Lambda_0\colon E\times_p B^{I}\to E^{I}$ defined by
\[\Lambda_0(e,\gamma)(t)=\Lambda(e,\tilde{\gamma})(t\varepsilon)\] 
is continuous. A direct computation shows that $\Lambda_0$ is a path-lifting map for $p$.

Now, if $\gamma$ is a constant path in $B$ and $e\in p^{-1}(\gamma(0))$, then $\tilde{\gamma}$ is a constant path and hence $\Lambda(e,\tilde{\gamma})(t)=e$ for all $t\in[0,\varepsilon]$. It follows that $\Lambda_0(e,\gamma)(t)=e$ for all $t\in I$. Hence, $\Lambda_0$ is a regular path-lifting map for $p$.
\end{proof}

\begin{ex}
In this example we will exhibit a Hurewicz fibration between finite T$_0$--spaces that does not have down beat points and that admits a path-lifting map which is not regular. This shows that a fibration without down beat points can have non-regular path-lifting maps.

Let $E$ be the finite T$_0$--space whose underlying set is $\{a,b,c\}$ and whose topology is $\{\varnothing,\{b\},\{c\},\{b,c\},\{a,b,c\}\}$. Let $B=\{*\}$ be the singleton and let $p\colon E\to B$ the only possible map. Note that $p$ is a Hurewicz fibration which does not have down beat points.

Let $\Lambda\colon E\times_p B^{I}\to E^{I}$ be defined by
\[
\Lambda(e,\gamma)(t)=\left\{
\begin{array}{ll}
e&\text{if $t<1$,}\\
a&\text{if $t=1$,}
\end{array}
\right.
\]
for all $e\in E$ and all $t\in I$, where $\gamma$ is the only possible path in $B$. It is easy to verify that $\Lambda$ is a path-lifting map for $p$ which is not regular.
\end{ex}

In subsection \ref{subs:fib_no_cerrada} we will exhibit a Hurewicz fibration that does not have up beat points and that does not have regular path-lifting maps.

Recall that if $\gamma$ is a path in a topological space $X$ and $t\in I$, the map $\gamma_{[0,t]}$ is the path in $X$ defined by $\gamma_{[0,t]}(s)=\gamma(st)$ for all $s\in I$ (definition \ref{defi:gamma_[0,t]}).

\begin{definition}
Let $E$ and $B$ be topological spaces and let $p\colon E\to B$ be a fibration. Let $\Lambda$ be a regular path-lifting map for $p$. We say that $\Lambda$ is a \emph{normalized} regular path-lifting map if $\Lambda(e,\gamma)_{[0,t]}=\Lambda(e,\gamma_{[0,t]})$ for all $t\in I$ and for all $(e,\gamma)\in E\times_p B^{I}$.
\end{definition}

\begin{lemma}\label{lemm:fib_reg_N} 
Let $E$ and $B$ be topological spaces and let $p\colon E\to B$ be a fibration which admits a regular path-lifting map $\Lambda$. Then there exists a normalized regular path-lifting map $N(\Lambda)$ for $p$.
\end{lemma}

\begin{proof}
We define $N(\Lambda)\colon E\times_p B^{I}\to E^{I}$ by
\[N(\Lambda)(e,\gamma)(t)=\Lambda(e,\gamma_{[0,t]})(1)\] 
for all $t\in I$ and for all $(e,\gamma)\in E\times_p B^{I}$. The assignment $(\gamma,t,s)\mapsto \gamma(st)$ of $B^{I}\times I\times I$ in $B$ is clearly continuous, since it can be factored as the composition 
\[B^{I}\times I\times I\xrightarrow{\id_{B^{I}}\times\mu}B^{I}\times I\xrightarrow{\ev}B\]
where $\mu\colon I\times I\to I$ is the multiplication map. By the exponential law, the assignment $(\gamma,t)\mapsto \gamma_{[0,t]}$ from $B^{I}\times I$ to $B^{I}$ is also continuous.

It is not difficult to prove that this map induces a continuous map 
\[(E\times_p B^{I})\times I\to E\times_p B^{I}\]
which sends the element $(e,\gamma,t)\in (E\times_p B^{I})\times I$ to $(e,\gamma_{[0,t]})$. Composing this map with 
\[E\times_p B^{I}\xrightarrow{\Lambda}E^{I}\xrightarrow{\ev_1}E,\]
we conclude that $N(\Lambda)^{\flat}$ is continuous. Thus $N(\Lambda)$ is also a continuous map.

Since $\Lambda$ is regular and $\gamma_{[0,0]}=C_{\gamma(0)}$ for all $\gamma\in B^{I}$, it follows that
\[N(\Lambda)(e,\gamma)(0)=\Lambda(e,\gamma_{[0,0]})(1)=\Lambda(e,C_{\gamma(0)})(1)=C_e(1)=e\]
for all $(e,\gamma)\in E\times_p B^{I}$. On the other hand,
\[pN(\Lambda)(e,\gamma)(t)=p\Lambda(e,\gamma_{[0,t]})(1)=\gamma_{[0,t]}(1)=\gamma(t)\]
for all $(e,\gamma)\in E\times_p B^{I}$ and $t\in I$. Thus, $N(\Lambda)$ is a path-lifting map for $p$.

Now, since for every $\gamma\in B^{I}$ and $s,t\in I$, $\left(\gamma_{[0,t]}\right)_{[0,s]}=\gamma_{[0,st]}$, it follows that
\begin{align*}
N(\Lambda)(e,\gamma)_{[0,t]}(s)&=N(\Lambda)(e,\gamma)(st)=\Lambda(e,\gamma_{[0,st]})(1)=\\
&=\Lambda\left(e,\left(\gamma_{[0,t]}\right)_{[0,s]}\right)(1)=N(\Lambda)(e,\gamma_{[0,t]})(s)
\end{align*}
for every $(e,\gamma)\in E\times_p B^{I}$ and $s,t\in I$. Thus $N(\Lambda)(e,\gamma)_{[0,t]}=N(\Lambda)(e,\gamma_{[0,t]})$ for all $t\in I$ and for all $(e,\gamma)\in E\times_p B^{I}$ as desired.
\end{proof}

\begin{rem}
Under the hypotheses of the previous lemma, we may consider the assignment $\Lambda\mapsto N(\Lambda)$ as an operator in the set of regular path-lifting maps of $p$. In this case, it is easy to verify that $N^{2}=N$.
\end{rem}

\begin{definition}
Let $B$ be a finite T$_0$--space and let $\gamma$ be a path in $B$. We say that $\gamma$ is a path of \emph{type $D$} if it is constant in the interval $(0,1]$ and we say that $\gamma$ is a path of \emph{type $U$} if it is constant in the interval $[0,1)$.
\end{definition}

In informal terms, the paths of type $D$ ``go down'' in time $t=0$ and remain constant in the interval $(0,1]$, while the paths of type $U$ remain constant in the interval $[0,1)$ and then ``go up'' in time $t=1$. We make these claims precise in the following remark.

\begin{rem}\label{rema:tipo_caminos}
Let $B$ be a finite T$_0$--space and let $\gamma$ be a path in $B$. It is easy to see that
\begin{enumerate}
\item if $\gamma$ is of type $D$, then $\gamma(0)\geq \gamma(t)=\gamma(1)$ for all $t\in (0,1]$, and 
\item if $\gamma$ is of type $U$, then $\gamma(0)=\gamma(t)\leq\gamma(1)$ for all $t\in [0,1)$. 
\end{enumerate}
Equivalently, $\gamma$ is of type $D$ if $\gamma(0)\geq \gamma(1)$ and $\gamma=\eta(\gamma(0)\geq \gamma(1))$ (cf. definition \ref{defi:eta}) and $\gamma$ is of type $U$ if $\gamma(0)\leq\gamma(1)$ and $\gamma=\eta(\gamma(0)\leq\gamma(1))$.
\end{rem}

\begin{lemma}\label{lemm:tipo_caminos}
Let $E$ and $B$ be finite T$_0$--spaces, let $p\colon E\to B$ be a fibration and let $\Lambda$ be a normalized regular path-lifting map for $p$. 
Let $\gamma$ be a path in $B$ and let $e\in p^{-1}(\gamma(0))$.
\begin{enumerate}
\item If $\gamma$ is of type $D$, then $\Lambda(e,\gamma)$ is of type $D$.
\item If $\gamma$ is of type $U$, then $\Lambda(e,\gamma)$ is of type $U$.
\end{enumerate}
\end{lemma}

\begin{proof}
Suppose that $\gamma$ is of type $D$. Then $\gamma=\gamma_{[0,t]}$ for all $t\in (0,1]$ and hence 
\[\Lambda(e,\gamma)(t)=\Lambda(e,\gamma)_{[0,t]}(1)=\Lambda(e,\gamma_{[0,t]})(1)=\Lambda(e,\gamma)(1)\]
for all $t\in (0,1]$. Hence, $\Lambda(e,\gamma)$ is of type $D$.

Now suppose that $\gamma$ is of type $U$. Then $\gamma_{[0,t]}=C_{\gamma(0)}$ for all $t\in [0,1)$. It follows that 
\[\Lambda(e,\gamma)(t)=\Lambda(e,\gamma)_{[0,t]}(1)=\Lambda(e,\gamma_{[0,t]})(1)=\Lambda(e,C_{\gamma(0)})(1)=C_e(1)=e\]
for all $t\in [0,1)$. Hence, $\Lambda(e,\gamma)$ is of type $U$.
\end{proof}

The following theorem is the principal result of this section.

\begin{theo}\label{theo:fib_reg_es_bif}
Let $E$ and $B$ be finite T$_0$--spaces and let $p\colon E\to B$ be a Hurewicz fibration which admits a regular path-lifting map. Then $p$ is a Grothendieck bifibration.
\end{theo}

\begin{proof}
We have to prove that
\begin{enumerate}
\item for all $e\in E$ and for all $b\leq p(e)$, the set $U_e\cap p^{-1}(U_b)$ has a maximum element which belongs to $p^{-1}(b)$, and
\item for all $e\in E$ and for all $b\geq p(e)$, the set $F_e\cap p^{-1}(F_b)$ has a minimum element which belongs to $p^{-1}(b)$.
\end{enumerate}
By \ref{lemm:fib_reg_N}, we may consider a normalized regular path-lifting map $\Lambda$ for $p$. 

First we will prove $(1)$. Let $e\in E$ and let $b\leq p(e)$. Note that if $p(e)=b$ the result follows trivially since in that case $e\in p^{-1}(b)$ and $e=\max (U_e\cap p^{-1}(U_{b}))$. Then, we may suppose that $b<p(e)$.

Let $X=\{e\}\cup (U_e \cap p^{-1}(U_b))$ and let $i\colon X\to E$ be the inclusion map. Let $f\colon X\to B$ be the map defined by 
\[
f(x)=\left\{
\begin{array}{ll}
p(x)&\text{if $x\neq e$,}\\
b&\text{if $x=e$.}
\end{array}
\right.
\]
It is easy to see that $f$ is continuous and that $f\leq pi$ in $B^{X}$. The path $\eta(pi\geq f)$ induces a homotopy $H\colon X\times I\to B$ from $pi$ to $f$ (cf. definition \ref{defi:eta}). Since $Hi_0=pi$, we may use the map $\Lambda$ to obtain a lift $\widetilde{H}$ of $H$ by $p$ from $i$, which is the map $\tilde{H}=(\Lambda\phi)^{\flat}$ induced by $\Lambda\phi$ by the exponential law, where $\phi\colon X\to E\times_p B^{I}$ is the induced map in the pullback by $i$ and $H^{\sharp}\colon X\to B^{I}$.

Note that, for all $x\in U_e\cap p^{-1}(U_b)$, the path $Hi_x$ is the constant path $C_{p(x)}$. By the regularity of $\Lambda$ we obtain that $\widetilde{H}i_x$ is the constant path $C_x$. On the other hand, the path $Hi_e$ is the path $\eta(p(e)\geq b)$ which is a path of type $D$. It follows from \ref{lemm:tipo_caminos} that $\widetilde{H}i_e$ is the path $\eta(e\geq e')$ of type $D$, where $e'=\widetilde{H}(e,1)$.

It is easy to see that $e'\in p^{-1}(b)$. Indeed,
\[p(e')=p\widetilde{H}(e,1)=H(e,1)=f(e)=b.\]
We will prove that $e'=\max(U_e\cap p^{-1}(U_b))$. Let $x\in U_e\cap p^{-1}(U_b)$. By the continuity of $\widetilde{H}i_1$ and since $x\leq e$, we obtain that
\[x=C_x(1)=\widetilde{H}i_x(1)=\widetilde{H}i_1(x)\leq\widetilde{H}i_1(e)=e'.\]
The result follows.

Now we will prove $(2)$. Let $e\in E$ and let $b>p(e)$. Let $X=\{e\}\cup (F_e \cap p^{-1}(F_b))$, let $i\colon X\to E$ be the inclusion map and let $f\colon X\to B$ be defined by
\[
f(x)=\left\{
\begin{array}{ll}
p(x)&\text{if $x\neq e$,}\\
b&\text{if $x=e$.}
\end{array}
\right.
\]
It is not difficult to prove that $f$ is continuous. Note that $f\geq pi$. Now, the path $\eta(pi\leq f)$ in $B^{X}$ induces a homotopy $H\colon X\times I\to B$ from $pi$ to $f$ and we may use $\Lambda$ to obtain a lift $\widetilde{H}$ of $H$ by $p$ from $pi$. Note that the path $\widetilde{H}i_x$ is the constant path $C_x$ for all $x\in F_e\cap p^{-1}(F_b)$. Note also that the path $Hi_e$ is the path $\eta(p(e)\leq b)$ which is a path of type $U$. It follows from \ref{lemm:tipo_caminos} that $\widetilde{H}i_e$ is the path $\eta(e\leq e')$ of type $U$, where $e'=\widetilde{H}(e,1)$. Thus $e'\in p^{-1}(b)$. The proof that $e'=\min(F_e\cap p^{-1}(F_b))$ is similar to the one that was done for the previous case.
\end{proof}

From the results given in this section, we obtain the following theorem. 

\begin{theo}\label{theo:fib_sin_dbp_bif_groth}
Let $E$ and $B$ be finite T$_0$--spaces and let $p\colon E\to B$ be a Hurewicz fibration without down beat points. Then
\begin{enumerate}
\item $p$ has a normalized regular path-lifting map.
\item $p$ is a Grothendieck bifibration.
\end{enumerate}
\end{theo}

\begin{proof}
Since $p$ does not have down beat points, it follows from \ref{theo:fib_sin_dbp_regular} and \ref{lemm:fib_reg_N} that there exists a normalized regular path-lifting map for $p$. This proves $(1)$, while proposition $(2)$ follows from $(1)$ and \ref{theo:fib_reg_es_bif}.
\end{proof}

The following result is immediate.

\begin{coro}\label{coro:fib_tiene_dbpr_bif}
Let $E$ and $B$ be finite T$_0$--spaces and let $p\colon E\to B$ be a Hurewicz fibration. Then the smallest dbp--retract of $p$ is a Grothendieck bifibration.
\end{coro}

As a corollary of the previous results, we obtain a combinatorial proof of the fact that the fibers of a fibration with connected codomain are homotopy equivalent for the particular case of Hurewicz fibrations between finite T$_0$--spaces.

\begin{coro}
Let $E$ and $B$ be finite T$_0$--spaces such that $B$ is connected and let $p\colon E\to B$ be a Hurewicz fibration. Then the fibers of $p$ are homotopy equivalent.
\end{coro}

\begin{proof}
Let $p_0$ be the smallest dbp--retract of $p$. Then $p_0$ is a Hurewicz fibration without down beat points. It follows from \ref{theo:fib_sin_dbp_bif_groth} that $p_0$ is a Grothendieck bifibration. Since $B$ is connected, it follows from  \ref{prop_alpha_y_beta_fib_hom_eq} that the fibers of $p_0$ are homotopy equivalent. But for each $b\in B$, the fiber $p_0^{-1}(b)$ is a dbp--retract of the fiber $p^{-1}(b)$ since it can be obtained from it by successively removing down beat points of $p$. The result follows.
\end{proof}

\begin{lemma}
Let $E$ and $B$ be finite T$_0$--spaces and let $p\colon E\to B$ be a continuous map without down beat points. Suppose, in addition, that $B$ has a maximum element and that $h(B)=1$. Then the following propositions are equivalent.
\begin{enumerate}
\item $p$ is a Hurewicz fibration.
\item $p$ is a Grothendieck bifibration.
\item $p$ is a retract of the canonical projection $\pi_B\colon E\times B\to B$.
\end{enumerate}
\end{lemma}

\begin{proof}
The implication $(1)\Rightarrow (2)$ follows from \ref{theo:fib_sin_dbp_bif_groth}. From \ref{lemm:altura_1_con_maximo} it follows that $(2)\Rightarrow (3)$. The implication $(3)\Rightarrow (1)$ is clear.
\end{proof}

It follows from the previous lemma and from \ref{coro:dbpr_minimo_fibracion} that if $E$ is a finite T$_0$--space, $\S$ is the Sierpinski space and $p\colon E\to \S$ is a continuous map, then $p$ is a Hurewicz fibration if and only if the smallest dbp--retract of $p$ is a Grothendieck bifibration. The following theorem extends this result to other codomains which have a minimum element.

\begin{theo}\label{theo:bif_con_minimo_es_fib}
Let $E$ and $B$ be finite T$_0$--spaces and let $p\colon E\to B$ be a continuous map. Suppose, in addition, that $B$ has a minimum element $b_0$. Then $p$ is a Hurewicz fibration if and only if the smallest dbp--retract of $p$ is a Grothendieck bifibration.
\end{theo}

\begin{proof}
By corollaries \ref{coro:fib_tiene_dbpr_bif} and \ref{coro:dbpr_minimo_fibracion}, it suffices to prove that if $p$ is a Grothendieck bifibration, then it is a Hurewicz fibration. Thus suppose that $p$ is a Grothendieck bifibration. Consider the functors $\alpha$ and $\beta$ constructed in \ref{theo:p_alpha_beta}.

We define the map $s\colon E\to E^{B}$ by
\[s(e)(b)=\beta(b_0\leq b)\alpha(p(e)\geq b_0)(e)\]
for each $b\in B$ and for each $e\in E$.

We will prove that $s$ is well defined. If $b,b'\in B$ are such that $b\leq b'$, then $F_{b'}\subseteq F_{b}$. Thus $\beta(b_0\leq b)(e_0)\leq \beta(b_0\leq b')(e_0)$ for all $e_0\in p^{-1}(b_0)$ by remark \ref{rema:cleavage_for_opfibration}. Given $e\in E$ and taking $e_0=\alpha(p(e)\geq b_0)(e)$ in the previous inequality, it follows that $s(e)(b)\leq s(e)(b')$. Hence, $s(e)$ is continuous for all $e\in E$. Therefore $s$ is well defined as desired.

Now we will prove that $s$ is continuous. Let $e,e'\in E$ be such that $e\leq e'$. Then $U_e\subseteq U_{e'}$ and thus
\[\alpha(p(e)\geq b_0)(e)=\max(U_e\cap p^{-1}(b_0))\leq \max(U_{e'}\cap p^{-1}(b_0))=\alpha(p(e')\geq b_0)(e').\]
Now, if $b\in B$, from the continuity of $\beta(b_0\leq b)$ it follows that 
\[s(e)(b)=\beta(b_0\leq b)\alpha(p(e)\geq b_0)(e)\leq\beta(b_0\leq b)\alpha(p(e')\geq b_0)(e')=s(e')(b).\]
Hence, $s(e)\leq s(e')$. 

Observe that the maps $s\colon E\to E^{B}$ and $\ev\colon B^{I}\times I\to B$ induce a continuous map 
\[(E\times_p B^{I})\times I\to E^{B}\times B\]
that sends each element $(e,\gamma,t)$ of $(E\times_p B^{I})\times I$ to the element $(s(e),\gamma(t))$ of $E^{B}\times B$. Composing this map with the evaluation map $\ev\colon E^{B}\times B\to B$ we obtain a continuous map
$\phi\colon (E\times_p B^{I})\times I\to E$ defined by $\phi(e,\gamma,t)=s(e)(\gamma(t))$ for all $(e,\gamma)\in E\times_p B^{I}$ and all $t\in I$.

We define now $\lambda\colon (E\times_p B^{I})\times I\to E$ by 
\[\lambda(e,\gamma,t)=\left\{
\begin{array}{ll}
e&\text{if $t=0$,}\\
s(e)(\gamma(t))&\text{if $t>0$.}
\end{array}
\right.                                                        
\]
It is clear that $\lambda(e,\gamma,t)=\phi(e,\gamma,t)$ for all $(e,\gamma,t)$ in the open subset $(E\times_p B^{I})\times (0,1]$. On the other hand, if $(e,\gamma)\in E\times_p B^{I}$, we have that
\begin{align*}
\phi(e,\gamma,0)&=s(e)(\gamma(0))=\beta(b_0\leq \gamma(0))\alpha(p(e)\geq b_0)(e)=\\
&=\beta(b_0\leq \gamma(0))\alpha(\gamma(0)\geq b_0)(e)\leq e=\lambda(e,\gamma,0)
\end{align*}
where the inequality holds by item $(3)$ of \ref{lemm:props_de_alpha_beta}. Thus, $\lambda\geq \phi$. In addition, it is clear that the restriction of $\lambda$ to the closed subset $E\times_p B^{I}\times\{0\}$ is continuous since it coincides with the projection to $E$. By \ref{prop:f_leq_g_g_cont}, it follows that $\lambda$ is continuous. On the other hand,  we have that $\lambda(e,\gamma,0)=e$ for all $(e,\gamma)\in E\times_p B^{I}$ and since $s(e)$ is a section of $p$ for all $e\in E$ it follows that $p\lambda(e,\gamma,t)=\gamma(t)$ for all $(e,\gamma)\in E\times_p B^{I}$ and all $t\in I$. Hence, $\lambda^\sharp\colon E\times_p B^{I}\to E^{I} $ is a path-lifting map for $p$. Therefore, $p$ is a Hurewicz fibration.
\end{proof}

\section{Examples}

In this section we will show several important examples which give counterexamples to natural questions related to this work.

\subsection{A Hurewicz fibration that is not a closed map}
\label{subs:fib_no_cerrada}

In this example we exhibit a Hurewicz fibration $p_1$ between finite T$_0$--spaces that is not closed. As a consequence we obtain an example of a Hurewicz fibration between finite T$_0$--spaces whose opposite map is not a Hurewicz fibration.

Let $E_1$ and $B_1$ be the finite T$_0$--spaces represented in the following diagram and let $p_1\colon E_1\to B_1$ be defined by $p_1(x,t)=x$.

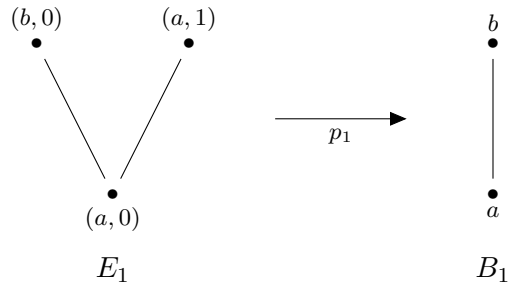
\begin{center}
\begin{tikzpicture}[x=2cm,y=2cm]
\tikzstyle{every node}=[font=\footnotesize]
\draw (0,1) node (b0){$\bullet$} node[above=1]{$(b,0)$};
\draw (0.5,0) node (a0){$\bullet$} node[below=1]{$(a,0)$};
\draw (1,1) node (a1){$\bullet$} node[above=1]{$(a,1)$};
\draw (b0)--(a0);
\draw (a0)--(a1);
\draw (1.5,0.5) node (s){};
\draw (2.5,0.5) node (t){};
\draw[\arr] (s) -> (t) node [midway,below] {$p_1$};
\draw (3,0) node (a){$\bullet$} node [below=1]{$a$};
\draw (3,1) node (b){$\bullet$} node [above=1]{$b$};
\draw (a)--(b);
\draw (0.5,-0.5) node (E){\normalsize $E_1$};
\draw (3,-0.5) node (B){\normalsize $B_1$};
\end{tikzpicture}
\end{center}

Note that if we remove the down beat point $(a,1)$ of $p_1$ we obtain a homeomorphism, and in particular, a Hurewicz fibration. From \ref{theo:fib_al_sacar_dbp}, it follows that $p_1$ is a Hurewicz fibration. Note also that, since $p_1(\{(a,1)\})=\{a\}$, then $p_1$ is not closed. Hence, the Hurewicz fibrations between finite T$_0$--spaces are not necessarily closed.

Observe also that the space $B_1$ has up beat points although the space $E_1$ does not have up beat points. This shows that \ref{coro:E_sin_dbp_entonces_B_sin_dbp} does not hold if we replace down beat points by up beat points.

Since $p_1$ is not a closed map, then $p_1^{\op}$ is not an open map. Thus, from \ref{coro:p_abta} it follows that $p_1^{\op}$ is not a fibration.

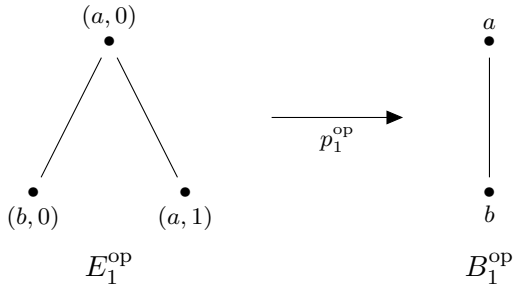
\begin{center}
\begin{tikzpicture}[x=2cm,y=2cm]
\tikzstyle{every node}=[font=\footnotesize]
\draw (0,0) node (b0){$\bullet$} node[below=1]{$(b,0)$};
\draw (0.5,1) node (a0){$\bullet$} node[above=1]{$(a,0)$};
\draw (1,0) node (a1){$\bullet$} node[below=1]{$(a,1)$};
\draw (b0)--(a0);
\draw (a0)--(a1);
\draw (1.5,0.5) node (s){};
\draw (2.5,0.5) node (t){};
\draw[\arr] (s) -> (t) node [midway,below] {$p_1^{\op}$};
\draw (3,1) node (a){$\bullet$} node [above=1]{$a$};
\draw (3,0) node (b){$\bullet$} node [below=1]{$b$};
\draw (a)--(b);
\draw (0.5,-0.5) node (E){\normalsize $E_1^\op$};
\draw (3,-0.5) node (B){\normalsize $B_1^\op$};
\end{tikzpicture}
\end{center}

Note that if we remove the up beat point $(a,1)$ of $p_1^{\op}$ we obtain a homeomorphism. This shows that in \ref{theo:fib_al_sacar_dbp} it is essential that the beat point that is removed is a down beat point of the considered map.

Moreover, the map $p_1$ does not have up beat points, and as $p_1$ is not a Grothendieck bifibration, it follows from  \ref{theo:fib_reg_es_bif} that $p_1$ does not have regular path-lifting maps. This shows that item $(2)$ of  \ref{theo:fib_sin_dbp_regular} does not hold changing down beat points to up beat points.

\subsection{A Serre fibration that is not a Hurewicz fibration}

In this example we will show that there exist maps between finite T$_0$--spaces that have the homotopy lifting property with respect to metric spaces (and in particular, with respect to locally finite CW--complexes) which are not Hurewicz fibrations.

Let $E_2$ and $B_2$ be the finite T$_0$--spaces represented in the following diagram and let $p_2\colon E_2\to B_2$ be defined by $p_2(x,t)=x$.

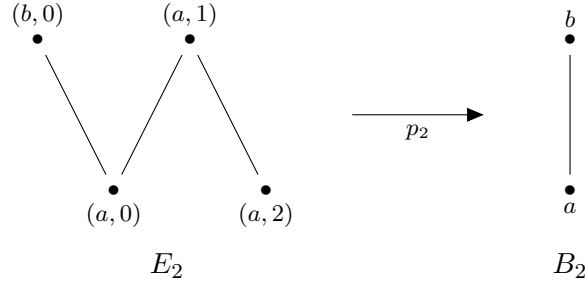
\begin{center}
\begin{tikzpicture}[x=2cm,y=2cm]
\tikzstyle{every node}=[font=\footnotesize]
\draw (0,1) node (b0){$\bullet$} node[above=1]{$(b,0)$};
\draw (0.5,0) node (a0){$\bullet$} node[below=1]{$(a,0)$};
\draw (1,1) node (a1){$\bullet$} node[above=1]{$(a,1)$};
\draw (1.5,0) node (a2){$\bullet$} node[below=1]{$(a,2)$};
\draw (b0)--(a0);
\draw (a0)--(a1);
\draw (a1)--(a2);
\draw (2,0.5) node (s){};
\draw (3,0.5) node (t){};
\draw[\arr] (s) -> (t) node [midway,below] {$p_2$};
\draw (3.5,0) node (a){$\bullet$} node [below=1]{$a$};
\draw (3.5,1) node (b){$\bullet$} node [above=1]{$b$};
\draw (a)--(b);
\draw (0.85,-0.5) node (E){\normalsize $E_2$};
\draw (3.5,-0.5) node (B){\normalsize $B_2$};
\end{tikzpicture}
\end{center}

Note that the map $p_2$ does not have down beat points. Since $p_2(\{(a,1)\})=\{a\}$, then $p_2$ is not a closed map and it follows from \ref{prop:fin_sin_dbp_cerrada} that $p_2$ is not a Hurewicz fibration. Naturally, we could have arrived to the same conclusion from \ref{theo:Fe_cap_fibra_b} observing that there does not exist a point in the fiber of $a$ which is smaller than $(a,2)$ and whose closure intersects the fiber of $b$. We will see, however, that the map $p_2$ has the homotopy lifting property with respect to metric spaces and with respect to compact spaces. This implies that the map $p_2$ has the homotopy lifting property with respect to locally finite CW--complexes and with respect to finite topological spaces.

\begin{lemma}
Let $X$ be a topological space. If $X$ is either a metric space or a compact space, then the map $p_2$ has the homotopy lifting property with respect to $X$.
\end{lemma}

\begin{proof}
Let $f\colon X\to E_2$ and $H\colon X\times I\to B_2$ be continuous maps such that $Hi_0=p_2f$.

First suppose that $X$ is a metric space. Note that $X\times I$ is also a metric space and hence it is a normal space. Since $f^{-1}(F_{(a,2)})$ is a closed subset of $X$, we have that $f^{-1}(F_{(a,2)})\times \{0\}$ is a closed subset of $X\times I$ which is contained in the open subset $H^{-1}(a)$. Since $X\times I$ is a normal space, there exists an open subset $V$ of $X\times I$ such that 
\[f^{-1}(F_{(a,2)})\times \{0\}\subseteq V\subseteq\overline{V}\subseteq H^{-1}(a).\]

It is not difficult to verify that
\begin{enumerate}
\item The set $H^{-1}(b)$ is closed in $X\times I$.
\item The set $(f^{-1}(a,2)\times I)\cap V$ is open in $X\times I$.
\item The set $(f^{-1}(F_{(a,2)})\times I)\cap \overline{V}-(f^{-1}(a,2)\times I)\cap V$ is closed in $X\times I$.
\item The set $H^{-1}(a)-(f^{-1}(F_{(a,2)})\times I)\cap \overline{V}$ is open in $X\times I$.
\end{enumerate}
Moreover, these four sets are pairwise disjoint and cover $X\times I$.

We define the map $\tilde{H}\colon X\times I\to E_2$ by 
\[\tilde{H}(x,t)=
\left\{
\begin{array}{ll}
(b,0)&\text{if $(x,t)\in H^{-1}(b)$,}\\
(a,0)&\text{if $(x,t)\in H^{-1}(a)-(f^{-1}(F_{(a,2)})\times I)\cap \overline{V}$,}\\
(a,1)&\text{if $(x,t)\in (f^{-1}(F_{(a,2)})\times I)\cap \overline{V}-(f^{-1}(a,2)\times I)\cap V$, and}\\
(a,2)&\text{if $(x,t)\in (f^{-1}(a,2)\times I)\cap V$.}
\end{array}
\right.
\]
It is immediate that $\tilde{H}^{-1}(a,0)$ and $\tilde{H}^{-1}(a,2)$ are open subsets of $X\times I$. Moreover, since $\tilde{H}^{-1}(b,0)$ is closed, it follows that $\tilde{H}^{-1}(U_{(a,1)})$ is open. Finally, $\tilde{H}^{-1}(\{(a,1),(a,2)\})=(f^{-1}(F_{(a,2)})\times I)\cap \overline{V}$ and hence, it is closed in $X\times I$. It follows that $\tilde{H}^{-1}(U_{(b,0)})$ is open. Therefore, $\tilde{H}$ is continuous. It is easy to verify that $\tilde{H}i_0=f$ and that $p_2\tilde{H}=H$.

Now suppose that $X$ is compact. Then $f^{-1}(F_{(a,2)})$ is a compact subspace. Since $f^{-1}(F_{(a,2)})\times \{0\}\subseteq H^{-1}(a)$, by the Tube Lemma there exists $\varepsilon>0$ such that $f^{-1}(F_{(a,2)})\times [0,\varepsilon]\subseteq H^{-1}(a)$.

In this case, we define the map $\tilde{H}\colon X\times I\to E_2$ by 
\[\tilde{H}(x,t)=
\left\{
\begin{array}{ll}
(b,0)&\text{if $(x,t)\in H^{-1}(b)$,}\\
(a,0)&\text{if $(x,t)\in H^{-1}(a)-f^{-1}(F_{(a,2)})\times [0,\varepsilon]$,}\\
(a,1)&\text{if $(x,t)\in f^{-1}(F_{(a,2)})\times [0,\varepsilon]-f^{-1}(a,2)\times [0,\varepsilon)$, and}\\
(a,2)&\text{if $(x,t)\in f^{-1}(a,2)\times [0,\varepsilon)$.}
\end{array}
\right.
\]
As in the previous case, it is not difficult to prove that $\tilde{H}$ is a continuous map and that $\tilde{H}i_0=f$ and that $p_2\tilde{H}=H$.
\end{proof}

Since locally finite CW--complexes are metric spaces, we immediately obtain the following corollary.

\begin{coro}
The map $p_2$ is a Serre fibration that is not a Hurewicz fibration.
\end{coro}

\subsection{A Grothendieck bifibration that is not a retract of a fiber bundle}
\label{subs:no_retracto_de_proy}

In this example we exhibit a Grothendieck bifibration between finite T$_0$--spaces that is not a retract of a fiber bundle of the same type. In particular, it can not be obtained by successively removing beat points of a fiber bundle between finite T$_0$--spaces.

Let $E_3$ and $B_3$ be the finite T$_0$--spaces represented in the following diagram and let $p_3\colon E_3\to B_3$ be defined by $p_3(x,t)=x$.

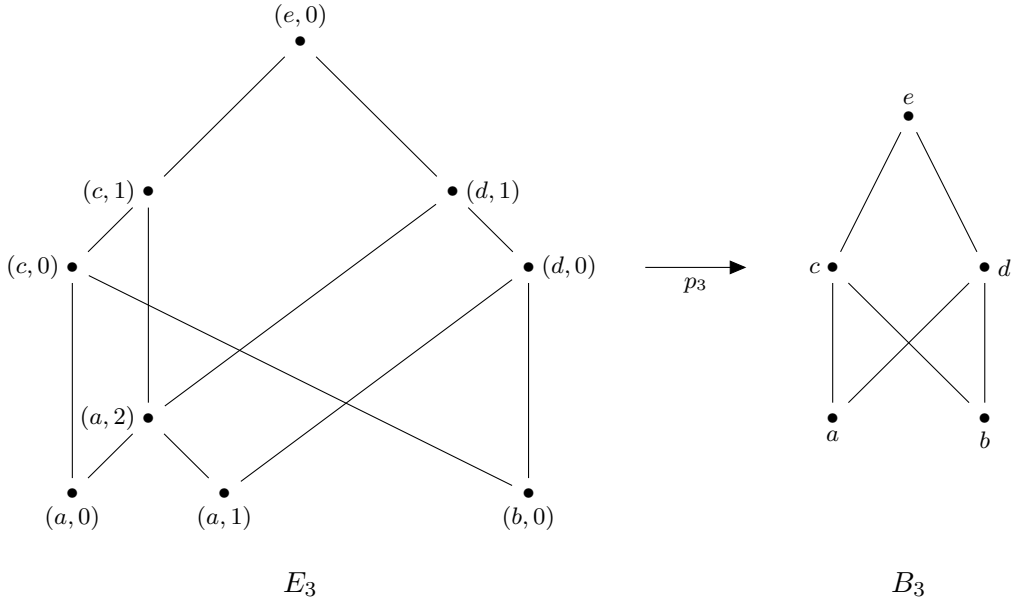
\smallskip

\begin{center}
\begin{tikzpicture}[x=2cm,y=2cm]
\tikzstyle{every node}=[font=\footnotesize]
\draw (0,0) node (a0){$\bullet$} node[below=1]{$(a,0)$};
\draw (1,0) node (a1){$\bullet$} node[below=1]{$(a,1)$};
\draw (0.5,0.5) node (a2){$\bullet$} node[left=1]{$(a,2)$};
\draw (3,0) node (b0){$\bullet$} node[below=1]{$(b,0)$};
\draw (0,1.5) node (c0){$\bullet$} node[left=1]{$(c,0)$};
\draw (0.5,2) node (c1){$\bullet$} node[left=1]{$(c,1)$};
\draw (3,1.5) node (d0){$\bullet$} node[right=1]{$(d,0)$};
\draw (2.5,2) node (d1){$\bullet$} node[right=1]{$(d,1)$};
\draw (1.5,3) node (e0){$\bullet$} node[above=1]{$(e,0)$};
\draw (a0)--(a2);
\draw (a0)--(c0);
\draw (a1)--(a2);
\draw (a1)--(d0);
\draw (a2)--(c1);
\draw (a2)--(d1);
\draw (b0)--(c0);
\draw (b0)--(d0);
\draw (c0)--(c1);
\draw (d0)--(d1);
\draw (c1)--(e0);
\draw (d1)--(e0);
\draw (3.7,1.5) node (s){};
\draw (4.5,1.5) node (t){};
\draw[\arr] (s) -> (t) node [midway,below] {$p_3$};
\draw (5,0.5) node (a){$\bullet$} node [below=1]{$a$};
\draw (6,0.5) node (b){$\bullet$} node [below=1]{$b$};
\draw (5,1.5) node (c){$\bullet$} node [left=1]{$c$};
\draw (6,1.5) node (d){$\bullet$} node [right=1]{$d$};
\draw (5.5,2.5) node (e){$\bullet$} node [above=1]{$e$};
\draw (a)--(c);
\draw (a)--(d);
\draw (b)--(c);
\draw (b)--(d);
\draw (c)--(e);
\draw (d)--(e);
\draw (1.5,-0.6) node (E){\normalsize $E_3$};
\draw (5.5,-0.6) node (B){\normalsize $B_3$};
\end{tikzpicture}
\end{center}

Observe that $p_3$ is a Grothendieck bifibration. In particular, the restriction 
\[p_3|\colon p_3^{-1}(\widehat{U}_e)\to \widehat{U}_e\] 
is also a Grothendieck bifibration.

We will prove that $p_3$ is not a retract of a projection associated to a product of finite T$_0$--spaces. This will show that the hypothesis over the height of the base of $p$ in \ref{lemm:altura_1_con_maximo} is necessary.

Let $X$ and $Y$ be finite T$_0$--spaces, let $\pi_X\colon X\times Y\to X$ and $\pi_Y\colon X\times Y\to Y$ be the projection maps and suppose that $p_3$ is a retract of $\pi_X$. Then, there exist continuous maps $i\colon E_3\to X\times Y$, $r\colon X\times Y\to E_3$, $j\colon B_3\to X$, $s\colon X\to B_3$ such that $ri=\id_{E_3}$, $sj=\id_{B_3}$, $\pi_X i=jp_3$ and $p_3 r=s\pi_X$, as the following commutative diagram shows.

\begin{center}
\begin{tikzpicture}[x=2cm,y=2cm]
\tikzstyle{every node}=[font=\footnotesize]
\draw (0,0) node (BL){$B_3$};
\draw (1,0) node (X){$X$};
\draw (2,0) node (BR){$B_3$};
\draw (0,1) node (EL){$E_3$};
\draw (1,1) node (XY){$X\times Y$};
\draw (2,1) node (ER){$E_3$};
\draw[\arr] (BL) -> (X) node [midway,below] {$j$};
\draw[\arr] (X) -> (BR) node [midway,below] {$s$};
\draw[\arr] (EL) -> (XY) node [midway,above] {$i$};
\draw[\arr] (XY) -> (ER) node [midway,above] {$r$};
\draw[\arr] (EL) -> (BL) node [midway,left] {$p_3$};
\draw[\arr] (XY) -> (X) node [midway,left] {$\pi_X$};
\draw[\arr] (ER) -> (BR) node [midway,right] {$p_3$};
\draw[\arr,bend left=45] (EL) to (ER);
\draw[\arr,bend right=45] (BL) to (BR);
\draw (1,1.63) node {$\id_{E_3}$};
\draw (1,-0.63) node {$\id_{B_3}$};
\end{tikzpicture}
\end{center}

Now, since $a\leq c$ then 
\[j(a)\leq j(c)=jp_3(c,0)=\pi_X i(c,0),\]
and since $(b,0)\leq (c,0)$ then $\pi_Yi(b,0)\leq \pi_Yi(c,0)$. It follows that
\[(j(a),\pi_Yi(b,0))\leq (\pi_Xi(c,0),\pi_Yi(c,0))=i(c,0).\]
Therefore $r(j(a),\pi_Yi(b,0))\leq ri(c,0)=(c,0)$. Similarly, $r(j(a),\pi_Yi(b,0))\leq (d,0)$. 

On the other hand,
\[p_3 r(j(a),\pi_Yi(b,0))=s\pi_X(j(a),\pi_Yi(b,0))=sj(a)=a\]
and hence $r(j(a),\pi_Yi(b,0))\in p_3^{-1}(a)$. Then $r(j(a),\pi_Yi(b,0))\in U_{(c,0)}\cap U_{(d,0)}\cap p_3^{-1}(a)=\varnothing$, which entails a contradiction. Therefore the map $p_3$ is not a retract of the projection of a product of finite T$_0$--spaces.

In a similar way it can be proved that $p_3|$ is not a retract of the projection of a product of finite T$_0$--spaces, which shows that the hypothesis that $B$ has a maximum element in \ref{lemm:altura_1_con_maximo} is also necessary. 

Note that since $B_3$ has a maximum element then, by \ref{coro_fiber_bundle_with_simply_connected_base}, all fiber bundles with base $B_3$ and fiber T$_0$ are isomorphic to projections of products. Hence, $p_3$ is not a retract of a fiber bundle over $B_3$ whose total space is finite and T$_0$. In particular, $p_3$ can not be obtained by successively removing beat points of a fiber bundle between finite T$_0$--spaces as we wanted to show.

\bibliographystyle{acm}
\bibliography{references}

\end{document}